\documentclass{article}
\usepackage{graphicx} 
\graphicspath{ {./Images/} }
\usepackage{subcaption}
\usepackage{amsmath,amsfonts,amsthm,amssymb}
\usepackage{hyperref}
\usepackage{enumitem}
\usepackage[toc,page]{appendix}
\usepackage[noadjust,compress]{cite}

\usepackage[margin=1in]{geometry}
\usepackage{xcolor}
\usepackage{algorithm}
\usepackage{algpseudocode}
\usepackage{comment}
\usepackage{tabularx}
\usepackage{multirow}

\newtheorem{theorem}{Theorem}

\newtheorem{lemma}{Lemma\textbf{}}

\theoremstyle{remark}
\newtheorem{remark}{Remark}
\newtheorem{example}{Example\textbf{}}

\DeclareMathOperator\erf{erf}

\newtheorem*{manualassumpinner}{Assumption}

\newenvironment{assump}[1]{%
  \begin{manualassumpinner}[{\normalfont#1}]%
}{%
  \end{manualassumpinner}%
}

\title{Discontinuity Analysis and Semi-Analytic Spectral Approximation for the Nonlocal Poisson Equation
  \thanks{This project is based upon work supported by the National Science Foundation under Grant No. 2108588.}
}

\author{%
  Thinh Dang\thanks{tdang@ksu.edu, Department of Mathematics, Kansas State University, Manhattan, KS.} \quad
  Bacim Alali\thanks{bacimalali@math.ksu.edu, Department of Mathematics, Kansas State University, Manhattan, KS.} \quad
  Nathan Albin\thanks{albin@ksu.edu, Department of Mathematics, Kansas State University, Manhattan, KS.}
}

\date{}

\begin{document}

\maketitle

\begin{abstract}
We study a  nonlocal Poisson problem with discontinuous source term and analyze how the regularity of the integral kernel determines the discontinuity structure of the corresponding solution. Under general assumptions on compactly supported integrable kernels, we show that jump discontinuities in the source term are inherited by the solution. We then identify two principal mechanisms governing higher-order regularity: singular behavior of the kernel at the origin and jump discontinuities of the kernel, or of its derivatives, at the horizon endpoints. Singularities at the origin lead to blow-up of certain derivatives of the solution at the source discontinuity, while jumps at the horizon generate cascades of derivative discontinuities at translated locations. These phenomena occur for kernels commonly used in peridynamic-type models. By contrast, compactly supported \(C^\infty\) kernels do not generate derivative blow-up or cascading losses of regularity, and in this case the source term and the solution have equivalent piecewise smooth regularity.

Motivated by this analysis, we develop a semi-analytic spectral method for the accurate numerical treatment of discontinuous nonlocal problems. The method uses successive smoothing transformations and explicitly constructed correction functions to convert the original problem into an auxiliary problem with improved regularity. A spectral solver is then applied to the smoothed problem, and the approximation to the original solution is recovered by adding back the analytic corrections. Numerical experiments show substantial gains in accuracy and convergence, demonstrating that the method effectively mitigates the loss of accuracy caused by discontinuities and Gibbs oscillations while retaining the efficiency of spectral methods.   
\end{abstract}
\section{Introduction}

We study a one-dimensional nonlocal Poisson problem with discontinuous source and investigate both the analytical structure of its solutions and the design of accurate numerical methods for resolving the resulting discontinuities. Specifically, we consider
\begin{align}\label{eq:Poisson}
    \begin{cases}
        Lu=f & \text{in } \Omega,\\
        u=b & \text{in } \Omega_I,
    \end{cases}
\end{align}
where \(b\) is a prescribed constraint on the interaction region, and the source term \(f\) has one or more jump discontinuities in \(\Omega\) while being otherwise differentiable. Here \(\Omega=(a_1,a_2)\) is a bounded open interval and
\[
\Omega_I=(a_1-\delta,a_1]\cup[a_2,a_2+\delta)
\]
denotes the associated \(\delta\)-collar region, where \(\delta>0\) is the horizon, representing the nonlocal interaction length scale. The nonlocal Laplace operator \cite{nonlocal_calc_2013} is defined by
\begin{align}\label{eq:L_operator}
    Lu(x)=\int_{\Omega\cup\Omega_I}(u(y)-u(x))K(x,y)\,dy,
\end{align}
where \(K\) is the interaction kernel.

Nonlocal diffusion and peridynamic models have attracted considerable attention because of their ability to represent long-range interactions and to naturally accommodate defects, interfaces, and fracture. In contrast to classical local models, nonlocal formulations incorporate interactions over a finite distance, which makes them especially suitable for describing processes in which discontinuities play an essential role. For the nonlocal Poisson problem \eqref{eq:Poisson}, well-posedness under standard assumptions on \(K\) is classical. For example, if \(f\in L^2(\Omega)\) and \(b\in L^2(\Omega_I)\), then there exists a unique solution
\(u\in L^2(\Omega\cup\Omega_I)\); see, for instance, \cite{radupLaplacian2012,du2012analysis}. Regularity, however, depends sensitively on both the kernel and the data. For integrable kernels, \cite{foss2016differentiability} proved that if \(f\in C^1(\overline{\Omega})\), then the corresponding solution belongs to \(W^{1,2}(\Omega)\).

The present work is motivated by the observation that discontinuous data arise naturally in nonlocal and peridynamic models, particularly in the study of fracture, interfaces, and material heterogeneities. In such situations, it is not enough to know that a solution exists: one must also understand how discontinuities in the source term influence the continuity, differentiability, and singular structure of the solution. This leads to the first objective of the paper, which is analytical in nature. Assuming that \(f\) is \(C^N\) away from its jump set, we determine how discontinuities propagate through the nonlocal operator and how the behavior of the kernel near the origin and at the horizon endpoints \(\pm\delta\) controls the regularity of the solution.
Our analysis identifies three principal mechanisms:
\begin{itemize}[leftmargin=*]
    \item \emph{Discontinuity inheritance:} under mild conditions ensuring that \(u*K\) is continuous, \(f\) is continuous at \(x\) if and only if \(u\) is continuous at \(x\) (Lemma~\ref{lemma:Cont}). Thus, jump discontinuities in the source are inherited by the solution.
    \item \emph{Singular behavior at the origin:} the behavior of the kernel at the origin determines whether higher derivatives of the solution remain finite near a jump. In particular, if \(K\), or one of its derivatives \(K^{(j)}\), is unbounded at the origin, then the corresponding derivative \(u^{(j+1)}\) blows up at the jump location (Theorems~\ref{thm:K_unbounded}--\ref{thm:unbounded_gen}).
    \item \emph{Horizon effects:} jump discontinuities of \(K\), or of one of its derivatives, at \(\pm\delta\) generate a cascade of derivative jumps in the solution at translated locations (Theorems~\ref{thm:K_jump}--\ref{thm:K_jump1}).
    
A related propagation of discontinuities to higher derivatives has been observed in the peridynamic literature; see,  \cite{silling2003deformation}.
\end{itemize}

It is important to emphasize that kernels unbounded at the origin and kernels with jump discontinuities at \(\pm\delta\) are both common in peridynamic models and  applications \cite{chen2016constructive,bobaru2016handbook,tian2013analysis,SelesonParks2011,du2012analysis}.

The second objective of the paper is computational. From the numerical point of view, a variety of methods have been developed for nonlocal problems of peridynamic type. The work in \cite{nonlocalnumerical} provides an extensive discussion of numerical methods, including finite element, finite difference, and spectral methods, for approximating solutions of nonlocal models. Other notable contributions to the computational treatment of peridynamics include the peridynamic differential operator method \cite{madenci2019peridynamic, oterkus2021peridynamic}, the boundary element method for peridynamics \cite{liang2021boundary}, meshfree methods \cite{silling2005meshfree, seleson2016convergence, sotudeh2025efficient}, and spectral methods \cite{du2016asymptotically, du2017fast, jafarzadeh2020efficient, jafarzadeh2022general, jafarzadeh2024perifast, mousavi2025fast, alali2020fourier, mustapha2025fourier}.

Modeling and resolving fracture and discontinuities is one of the key strengths of peridynamic models. At the same time, the presence of discontinuities requires special care in order to obtain accurate numerical approximations. It is well known that, in the presence of jump discontinuities, standard high-order methods lose their nominal convergence rates unless the discontinuity is explicitly incorporated into the discretization. In particular, Fourier and spectral approximations suffer from the Gibbs phenomenon, while standard finite difference and finite element discretizations for interface and jump problems typically reduce to first-order behavior near the discontinuity unless special interface-aware corrections are used \cite{gotlieb1997gibbs, leveque1994immersed, cao2017superconvergence}. Several studies have therefore focused on finite element approximations for peridynamic models with jump discontinuities; see, for example, \cite{chengunzburger2011, gunzburger2016, xugunzburger2016, xugunzburgerdu2016}. In particular, the work in \cite{chengunzburger2011} demonstrated that a discontinuous Galerkin method can achieve \(\mathcal{O}((\Delta x)^2)\) accuracy both when the location of the jump is known and, through iterative refinement, when it is unknown. However, finite element methods become inefficient when the grid size is much smaller than the horizon \(\delta\).

Spectral methods, by contrast, have demonstrated high accuracy and efficiency for smooth nonlocal problems. They have proven effective for smooth periodic problems in multiple dimensions \cite{alali2020fourier} and for smooth one- and two-dimensional problems on bounded domains \cite{mustapha2025fourier}. Moreover, the spectral methods developed in \cite{jafarzadeh2020efficient, jafarzadeh2022general, jafarzadeh2024perifast, mousavi2025fast} provide highly efficient computational tools for peridynamic and nonlocal models. Their appeal lies in the fact that, when the operator has a convolutional structure, the nonlocal equation can be treated efficiently in the frequency domain.

The main difficulty in applying spectral methods to fracture and discontinuity problems in peridynamics arises from two distinct sources. On the one hand, the Gibbs phenomenon degrades accuracy whenever Fourier-based methods are applied directly to discontinuous data. On the other hand, standard damage models do not always possess a convolutional structure and therefore do not reduce to simple pointwise multiplication in the frequency domain. The latter issue was addressed in \cite{jafarzadeh2022general}, where damage models based on pairwise bond failure were replaced by a local \emph{material integrity} indicator. It was shown in \cite{jafarzadeh2022general} that this formulation enables spectral solvers to reproduce benchmark fracture results accurately, thereby providing an elegant framework for dynamic fracture computations. Nevertheless, because the method still requires the application of the FFT to discontinuous data, its convergence rate remains limited by Gibbs oscillations.

To overcome this limitation, we introduce a framework that combines the regularity and discontinuity analysis developed in this work, as described in Sections~\ref{section:overview} and~\ref{section:discontinuity}, with an accurate and efficient semi-analytic spectral approach. The guiding idea is that the analytical description of the discontinuity structure should not only explain the behavior of solutions, but should also be incorporated directly into the numerical method.

Motivated by the regularity theory developed here, the computational part of the paper introduces a semi-analytic spectral framework for discontinuous nonlocal problems. The central idea is to transform the original problem into an auxiliary \emph{smoothed problem} by successively subtracting explicitly constructed correction functions that capture the discontinuity structure of the solution. The resulting smoothed problem has improved regularity and can therefore be solved much more accurately by a spectral method. The approximation to the original discontinuous solution is then reconstructed by combining the numerical solution of the smoothed problem with the analytic correction terms. In this way, the regularity analysis is used not only to explain the behavior of discontinuous solutions, but also to design an efficient computational method that restores high-order accuracy while retaining the efficiency of spectral solvers.

The remainder of the paper is organized as follows. In Section~\ref{section:overview}, we state the standing assumptions and derive a useful convolution representation of the operator. Section~\ref{section:discontinuity} contains the main regularity and discontinuity results, together with examples illustrating the sharpness of the theory for several commonly used kernels. In Section~\ref{section:semi-analytic}, we introduce the successive smoothing procedure and the resulting semi-analytic spectral method. Section~\ref{section:num} presents numerical experiments demonstrating the accuracy and convergence of the proposed approach.

\section{Preliminaries}\label{section:overview}

\subsection{Assumptions on the integral kernel and the source term}

In this section, we state the standing assumptions on the kernel \(K\) and the source term $f$. 
These
 assumptions generalize a large class of kernels commonly used in peridynamics and its applications to material fracture, see for example \cite{silling2007peridynamic,bobaru2016handbook,tian2013analysis,SelesonParks2011}. Concrete examples of kernels arising in peridynamic models and satisfying assumptions (A1)--(A2) below are given in Sections~\ref{section:discontinuity} and ~\ref{section:semi-analytic}; see \eqref{example_kernel_1}, \eqref{example_kernel_2},  \eqref{example_kernel_3}, \eqref{eq:K_numerical} and \eqref{K_smooth}.

\begin{assump}{A1} The kernel $K$ is a positive, radial and integrable kernel that is compactly supported. Specifically,
    \begin{align}\label{kernel}
    \begin{cases}
    K(x,y)=K(|x-y|)\ge 0, &\text{for all} \;x,y\in\Omega\cup\Omega_I,\\
    K(|x-y|)=0,&\text{for}\; |x-y|\ge \delta,\\
    K\in L^1(\mathbb{R}).
    \end{cases}
\end{align}
\end{assump}

In addition, we assume that \(K\) is smooth away from the origin and the horizon endpoints \(\pm\delta\).
\begin{assump}{A2}
   The kernel $K$ is continuous on $[-\delta,0)\cup (0,\delta]$ and $C^\infty$ on the intervals $(-\delta,0)$ and $(0,\delta)$.    
 \end{assump}

The behavior of $K$ at $0$ and $\pm\delta$ will be discussed in Sections ~\ref{section:discontinuity} and ~\ref{section:semi-analytic} where we will consider several cases of whether $K$ is bounded/unbounded and whether $K$ is continuous or has jumps at $\pm\delta$. 

Without loss of generality, we assume that the source term \(f\) has a jump discontinuity at \(0\) and is otherwise differentiable.
\begin{assump}{B}
\label{Data_assumptions} The source term $f$ has a jump discontinuity at $0$ and $f$ is $C^N$ on the intervals $(a_1,0)$ and $(0,a_2)$, for some $N\in\mathbb{N}$.  Moreover, $f$ and all of its derivatives, $f^{(j)}$ for $j=1,2,\cdots,N$, are  bounded on $(a_1,a_2)$.
\end{assump}
The case of multiple jump discontinuities can be treated similarly by applying the same analysis locally near each jump point. In particular, Assumption~(B) implies that $f\in L^2(\Omega)$.


\subsection{Convolution structure of the operator \texorpdfstring{$L$}{L}}\label{subsec:convolution}

Before analyzing discontinuities in the solution, it is useful to first examine the structure of the nonlocal operator when applied to functions with jump discontinuities. In particular, the operator admits a convolution-type representation that isolates the contribution of the kernel near singular points. This representation separates the effect of discontinuities in the source term from the regularity properties of the kernel itself, and it will play a central role in the analysis below.

This section develops this representation and establishes several auxiliary
results that will be used throughout the remainder of the paper.

Since \(K\in L^1(\mathbb{R})\), we use the decomposition
\begin{align}\label{eq:decomposition}
    Lu=u*K-\alpha u,
\end{align}
where \(\alpha:=\|K\|_{L^1(\mathbb{R})}\), and \(u\) is extended by zero outside \(\Omega\cup\Omega_I\).

\begin{lemma}\label{lemma:Cont}
Let $f\in L^2(\Omega)$ and $u\in L^2(\Omega\cup\Omega_I)$ be such that $L u=f$ in $\Omega$. Assume that $u\ast K\in C(\Omega)$. Then, $f$ is continuous at $x\in\Omega$ if and only if $u$ is continuous at $x$. 
\end{lemma}
\begin{proof}
    This is a direct application of the decomposition in \eqref{eq:decomposition}.
\end{proof}


Thus, whenever \(u*K\in C(\Omega)\), the solution \(u\) and the source \(f\) have exactly the same points of continuity and discontinuity.
The following are examples in which the hypothesis of the lemma is satisfied.
\begin{itemize}
    \item If \(K\in L^2(\mathbb{R})\) and \(u\in L^2(\Omega\cup\Omega_I)\), then \(u*K\in C(\Omega)\).
    \item If $K\in L^1(\mathbb{R})$ and $u\in L^\infty(\mathbb{R})$, then $u\ast K\in C(\Omega)$.
\end{itemize} 

The following  is a result from \cite{foss2016differentiability}.  
\begin{theorem}\label{thm:Radu_mod}
        Suppose that $K$ satisfies Assumption (A1). Let $u\in L^2(\Omega\cup\Omega_I)$ and $f\in L^2(\Omega)$ such that $L u=f$ in $\Omega$. 
        \begin{enumerate}
            \item If $f\in L^\infty(\Omega)$, then $u\in L^\infty(\Omega)$. 
            \item If $f\in W^{1,2}(\Omega)\cap L^\infty (\Omega)$, then $u\in W^{1,2}(\Omega)$.
        \end{enumerate}
\end{theorem}
Part (1) is Theorem 3.1 of \cite{foss2016differentiability}, and Part (2) is a mild relaxation of Theorem 4.3 therein. In Part (2), the original assumption is \(f\in C^1(\overline{\Omega})\); however, the same proof remains valid under the weaker hypothesis \(f\in W^{1,2}(\Omega)\cap L^\infty(\Omega)\).

For \(k\in\mathbb{N}\), define
\[
\varphi_k(x)=\frac{x^k}{k!}\chi_{[0,\infty)}(x), \qquad x\in\mathbb{R}.
\]
These functions serve as elementary building blocks for isolating jump contributions. Observe that
\begin{equation}\label{Phi}
    \varphi_{k+1}'(x)=\varphi_k(x), \qquad x\neq 0,
\end{equation}
and therefore, by \eqref{eq:decomposition},
\begin{equation}\label{Lphi}
    (L\varphi_{k+1})'(x)=L\varphi_k(x), \qquad x\neq 0.
\end{equation}
For a given kernel \(K\), the functions \(L\varphi_k\) can often be computed explicitly, and they will be used repeatedly in the sequel.

\begin{lemma}\label{lemma:Deri_K}
    Let $K$ satisfy  Assumption (A1). Then, 
    \begin{enumerate}
        \item $L\varphi_k\in L^\infty_{\text{loc}}(\mathbb{R})$ for all $k$.
        \item $(L\varphi_0)^\prime=K$ almost everywhere.
    \end{enumerate}
\end{lemma}
\begin{proof}
    \begin{enumerate}
        \item Observe that
    \begin{align}\label{eq:Deri_K_1}
        \begin{split}
            \varphi_k\ast K(x)=\int_{-\infty}^\infty \varphi_k(y)K(x-y)dy=\int_0^\infty \frac{y^k}{k!} K(y-x)dy=\int_{-x}^\infty \frac{(z+x)^k}{k!}K(z)dz\\
        =\int_{-x}^\delta \frac{(z+x)^k}{k!} K(z)dz.
        \end{split}     
    \end{align}
        Since $K\in L^1(\mathbb{R})$, it follows from \eqref{eq:Deri_K_1} that $\varphi_k\ast K\in L^\infty(\mathbb{R})$. Note that $\varphi_k\in L^\infty_{\text{loc}}(\mathbb{R})$ by definition. From the decomposition, $L\varphi_k=\varphi_k\ast K-\alpha \varphi_k$, we obtain that $L\varphi_k\in L^\infty_{\text{loc}}(\mathbb{R})$.
        \item In the case that $k=0$,
        \begin{align}\label{eq:Deri_K_2}
        \varphi_0\ast K(x)=\int_{-x}^\delta K(z)dz.
    \end{align}
        
        Since $K\in L^1(\mathbb{R})$, from the Fundamental theorem of Calculus, $(\varphi_0\ast K)^\prime=K$ almost everywhere. Note that $\varphi_0^\prime(x)=0$ for all $x\ne 0$.  Differentiating the identity \(L\varphi_0=\varphi_0\ast K-\alpha\varphi_0\) for \(x\neq 0\) completes the proof. 
    \end{enumerate}  
\end{proof}
We denote the jump of a function $\psi$ at $p$ by $[\psi]_p$, specifically,
\begin{align*}
    [\psi]_p=\lim_{x\to p^+}\psi(x)-\lim_{x\to p^-}\psi(x),
\end{align*}
given that the right and left limits exist.

\subsection{Differentiation under the integral}\label{subsection:diff}
The next two lemmas provide the differentiation tools needed in the proofs below. Lemma~\ref{lemma: Leibniz_easy} is the classical Leibniz rule on \(\mathbb{R}\), while Lemma~\ref{lemma:Leibniz} gives a measure-theoretic version suitable for differentiating convolution terms under minimal regularity assumptions.

\begin{lemma}[Leibniz's theorem on $\mathbb{R}$]\label{lemma: Leibniz_easy}
    Let $g(x,y)$ be a function such that both $g$ and $\frac{\partial g}{\partial x}$ are continuous in $x$ and $y$ in some $xy$-domain containing $\xi(x)\le y\le \zeta(x),\; c\le x\le d$. Suppose that $\xi(x)$ and $\zeta(x)$ are continuous and have continuous derivatives for $c\le x\le d$. Then for $c\le x\le d$,
    \begin{align*}
        \frac{d}{dx}\left(\int_{\xi(x)}^{\zeta(x)}g(x,y)dy\right)=g(x,\zeta(x))-g(x,\xi(x))+\int_{\xi(x)}^{\zeta(x)}\frac{\partial g}{\partial x}(x,y)dy
    \end{align*}
\end{lemma}

\begin{lemma}[Leibniz's theorem: Measure theory statement]\label{lemma:Leibniz}
    Let $X$ be an open set of $\mathbb{R}$ and $D$ be a measure space. Suppose that a function $g:X\times D\to \mathbb{R}$ satisfies the following conditions:
    \begin{enumerate}
        \item $g(x,\omega)$ is jointly measurable in 
$(x,\omega)$, and is integrable over $\omega$, for almost $x\in X$ held fixed.
        \item For almost $\omega\in D$, $g(x,\omega)$ is an absolutely continuous function of $x$.
        \item $\partial g/\partial x$ is locally integrable, that is for all compact intervals $[a,b]$ contained in $X$
        \begin{align*}
            \int_a^b \int_{D}\left|\frac{\partial g}{\partial x}(x,\omega)\right|d\omega dx<\infty.
        \end{align*}
    \end{enumerate}
    Then $\int_D g(x,\omega) d\omega$ is an absolutely continuous function of $x$, and for almost $x\in X$, its derivative exists and is given by 
    \begin{align*}
        \frac{d}{dx}\int_{D}g(x,\omega)d\omega=\int_D\frac{dg}{dx}(x,\omega)d\omega.
    \end{align*}
\end{lemma}

 .

\section{Effects of kernel regularity on the solution}
\label{section:discontinuity}

In this section, we investigate how the smoothness and singularity of the kernel influence the behavior of the solution near discontinuity points of the source term. We consider several cases, depending on whether the kernel \(K\) is bounded or unbounded, and whether \(K\) is continuous or has jump discontinuities at \(\pm\delta\). These distinctions are not merely technical. In particular, the kernel \(S\) in Example~1 is a commonly used peridynamic-type kernel, see for example \cite{silling2007peridynamic,silling,chen2016constructive,tian2013analysis,SelesonParks2011}, which is singular at the origin and discontinuous at the horizon. Thus, the kernel properties studied below arise naturally in models used in practice.


The results of this section show that these kernel features can have substantial consequences for solution regularity. For kernels that are singular at the origin, a jump discontinuity in the source term implies that either the source term itself or the solution must develop unbounded derivatives near the jump point. This is relevant both analytically and computationally, since it shows that discontinuous data may generate stronger local singularities than the classical local theory might suggest. Similarly, discontinuities of the kernel at the horizon \(\pm\delta\) produce a cascade of local drops in regularity at translated points \(\pm k\delta\). While this behavior follows naturally from the nonlocal structure of the operator, it may be viewed as an unintended modeling effect introduced by the particular kernel choice.

To motivate the different assumptions used throughout this section, we introduce several representative kernels that arise in peridynamic models.

\begin{example}
Let
\begin{align}\label{example_kernel_1}
    S(x,y)=S(|x-y|)=\frac{c_S}{|x-y|^\beta}\chi_{(-\delta,\delta)}(x-y),
    \qquad
    c_S=\frac{3-\beta}{\delta^{3-\beta}},
    \qquad 0<\beta<1.
\end{align}
This kernel is singular at \(x=0\) and has jump discontinuities at \(x=\pm\delta\); see Figure~\ref{fig:kernel_1}. It is representative of commonly used peridynamic kernels with truncated power-law behavior, see for instance \cite{chen2016constructive,bobaru2016handbook,tian2013analysis,SelesonParks2011}.
\end{example}

\begin{example}
Let
\begin{align}\label{example_kernel_2}
    P(x,y)=P(|x-y|)=c_P\left(1-\frac{|x-y|^\beta}{\delta^\beta}\right)\chi_{(-\delta,\delta)}(x-y),
    \qquad
    c_P=\frac{3(3+\beta)}{\beta\delta^3},
    \qquad 0<\beta<1.
\end{align}
This kernel, illustrated in Figure~\ref{fig:kernel_2}, is bounded and continuous, while its derivative is unbounded at \(x=0\):
\[
|P'(0^\pm)|=\infty.
\]
\end{example}

\begin{example}
Let
\begin{align}\label{example_kernel_3}
   G(x,y)=G(|x-y|)=c_G\exp\left(-\frac{|x-y|^2}{\delta^2}\right)\chi_{(-\delta,\delta)}(x-y),
   \qquad
   c_G=\frac{4e}{\delta^3\bigl(e\sqrt{\pi}\,\erf(1)-2\bigr)}.
\end{align}
This kernel, illustrated in Figure~\ref{fig:kernel_3}, is smooth on \(\mathbb{R}\setminus\{\pm\delta\}\) and has jump discontinuities at \(x=\pm\delta\).
\end{example}

\begin{example}
The following kernel belongs to \(C_c^\infty(\mathbb{R})\) and is supported on \([-\delta,\delta]\):
\begin{align}\label{K_smooth}
    Q(x,y)=Q(|x-y|)=c_Q\exp\left(-\frac{1}{1-|x-y|^2/\delta^2}\right)\chi_{(-\delta,\delta)}(x-y),
\end{align}
where
\[
c_Q=
\left(
\delta^3\int_0^1 r^2\exp\!\left(-\frac{1}{1-r^2}\right)\,dr
\right)^{-1}.
\]
\end{example}
\subsection{Unboundedness of the kernel and its higher derivatives}
\label{subsection:unboundedness}

We first consider the case in which the kernel itself is unbounded at the origin. In this setting, the singular behavior of \(K\) is transferred directly to the first derivative of the solution at the jump point.

\begin{theorem}\label{thm:K_unbounded}
Assume that the kernel \(K\) satisfies Assumptions (A1)--(A2), and, in addition, is singular at the origin and belongs to \(L^2(\mathbb{R})\).
\begin{enumerate}
    \item Suppose that \(f\) satisfies Assumption (B) and \(b\in W^{1,2}(\Omega_I)\). Then, for the solution \(u\in L^2(\Omega\cup\Omega_I)\) of \eqref{eq:Poisson}, one has
    \[
    |u'(0^\pm)|=\infty.
    \]
    \item Suppose that \(u\) satisfies Assumption (B). Then, for \(f:=Lu\), one has
    \[
    |f'(0^\pm)|=\infty.
    \]
\end{enumerate}
\end{theorem}

\begin{proof}
\begin{enumerate}
    \item
    Since \(u\in L^2(\mathbb{R})\) by zero extension and \(K\in L^2(\mathbb{R})\), the convolution \(u*K\) is continuous on \(\mathbb{R}\). From the decomposition
    \[
    u*K-\alpha u=f \qquad \text{in } \Omega,
    \]
    it follows that \(u\) has a jump discontinuity at \(0\) of magnitude
    \[
    c_0:=[u]_0=-\frac{1}{\alpha}[f]_0,
    \]
    where \(0<|c_0|<\infty\). Moreover, \(u\) is continuous on \(\Omega\setminus\{0\}\).

    Define
    \begin{equation}\label{eq:u0_thm_unbounded}
        u_0:=u-c_0\varphi_0.
    \end{equation}
    Then \(u_0\in C^0(\Omega)\) and \(u_0\) solves
    \begin{align}\label{eq:u0_solve_thm_unbounded}
        \begin{cases}
            Lu_0=f_0:=f-c_0L\varphi_0 & \text{in } \Omega,\\
            u_0=b-c_0\varphi_0 & \text{in } \Omega_I.
        \end{cases}
    \end{align}

    We now show that \(u_0\) has enough regularity to differentiate the convolution term. Since
    \(f_0=f-c_0L\varphi_0\), \(f\in L^\infty(\Omega)\) by Assumption (B), and \(L\varphi_0\in L^\infty(\Omega)\) by Lemma~\ref{lemma:Deri_K}, it follows that \(f_0\in L^\infty(\Omega)\). Moreover,
    \[
    f_0'=f'-c_0K
    \]
    almost everywhere in \(\Omega\). Since \(f'\) is  bounded on each side of the origin and \(K\in L^2(\mathbb{R})\), it follows that \(f_0'\in L^2(\Omega)\). Hence \(f_0\in W^{1,2}(\Omega)\cap L^\infty(\Omega)\). By Theorem~\ref{thm:Radu_mod}, we obtain \(u_0\in W^{1,2}(\Omega)\). Since \(b\in W^{1,2}(\Omega_I)\), it follows that
    \[
    u_0\in W^{1,2}(\Omega\cup\Omega_I).
    \]
    In particular, \(u_0\) and its weak derivative \(u_0'\) both belong to \(L^2(\Omega\cup\Omega_I)\).

    Consider
    \[
    (u_0*K)(x)=\int_{-\delta}^{\delta}u_0(x-y)K(y)\,dy
    =\int_{-\delta}^{\delta}g(x,y)\,dy,
    \]
    where \(g(x,y):=u_0(x-y)K(y)\). For each fixed \(x\in\Omega\), Hölder's inequality gives
    \[
    \|g(x,\cdot)\|_{L^1(-\delta,\delta)}
    \le \|u_0(x-\cdot)\|_{L^2(-\delta,\delta)}\|K\|_{L^2(\mathbb{R})}
    \le \|u_0\|_{L^2(\Omega\cup\Omega_I)}\|K\|_{L^2(\mathbb{R})}<\infty.
    \]
    For fixed \(y\in(-\delta,\delta)\), the pointwise derivative
    \[
    \frac{\partial g}{\partial x}(x,y)=u_0'(x-y)K(y)
    \]
    exists for almost every \(x\), and is integrable in \(x\) because \(u_0'\in L^2(\Omega\cup\Omega_I)\). Thus \(g\) is absolutely continuous in \(x\). Furthermore, for every compact interval \([a,b]\subset\Omega\),
    \[
    \int_a^b\int_{-\delta}^{\delta}\left|\frac{\partial g}{\partial x}(x,y)\right|\,dy\,dx
    =\int_a^b (|u_0'|*|K|)(x)\,dx
    \le (b-a)\|u_0'\|_{L^2(\Omega\cup\Omega_I)}\|K\|_{L^2(\mathbb{R})}<\infty.
    \]
    Therefore, by Lemma~\ref{lemma:Leibniz},
    \[
    \frac{d}{dx}(u_0*K)(x)=\int_{-\delta}^{\delta}\frac{\partial g}{\partial x}(x,y)\,dy
    =u_0'*K(x)
    \]
    for almost every \(x\in\Omega\). Since \(u_0'\in L^2(\mathbb{R})\) and \(K\in L^2(\mathbb{R})\), the convolution \(u_0'*K\) is continuous on \(\mathbb{R}\).

    On the other hand, from \eqref{eq:u0_solve_thm_unbounded} and Lemma~\ref{lemma:Deri_K},
    \[
    f_0'=f'-c_0K
    \]
    almost everywhere. Because \(K\) is unbounded at \(0\) and \(0<|c_0|<\infty\), we obtain
    \[
    |f_0'(0^\pm)|=\infty.
    \]
    Using the decomposition
    \[
    u_0*K-\alpha u_0=f_0,
    \]
    and the continuity of \((u_0*K)'\), it follows that
    \[
    |u_0'(0^\pm)|=\infty.
    \]
    Finally, since \(u=u_0+c_0\varphi_0\) by \eqref{eq:u0_thm_unbounded}, we conclude that
    \[
    |u'(0^\pm)|=\infty.
    \]

    \item
    As in Part (1), since \(u\in L^2(\mathbb{R})\) by zero extension and \(K\in L^2(\mathbb{R})\), the convolution \(u*K\) is continuous on \(\mathbb{R}\). From
    \[
    u*K-\alpha u=f \qquad \text{in } \Omega,
    \]
    it follows that \(f\) has a jump discontinuity at \(0\) of magnitude
    \[
    [f]_0=-\alpha [u]_0,
    \]
    and \(f\) is continuous on \(\Omega\setminus\{0\}\).

    Let
    \[
    c_0:=[u]_0=-\frac{1}{\alpha}[f]_0,
    \]
    and define \(u_0:=u-c_0\varphi_0\). Since \(u\) satisfies Assumption (B), it follows that \(u_0\) is absolutely continuous and hence belongs to \(W^{1,2}(\Omega\cup\Omega_I)\). Therefore, by Lemma~\ref{lemma:Leibniz},
    \[
    (u_0*K)'=u_0'*K
    \]
    almost everywhere. Since \(u_0'\in L^2(\mathbb{R})\) and \(K\in L^2(\mathbb{R})\), the function \((u_0*K)'\) is continuous on \(\mathbb{R}\).

    Let \(f_0:=f-c_0L\varphi_0\). From the decomposition
    \[
    u_0*K-\alpha u_0=f_0,
    \]
    it follows that
    \[
    |f_0'(0^\pm)|<\infty.
    \]
    On the other hand, Lemma~\ref{lemma:Deri_K} gives
    \[
    f_0'=f'-c_0K
    \]
    almost everywhere. Since \(K\) is unbounded at \(0\) and \(0<|c_0|<\infty\), we conclude that
    \[
    |f'(0^\pm)|=\infty.
    \]
\end{enumerate}
\end{proof}

Theorem~\ref{thm:K_unbounded} is illustrated by the  kernel $S$ given in \eqref{example_kernel_1}. 
A direct computation gives
\begin{align*}
    L_S\varphi_0(x)=
    \begin{cases}
        \dfrac{c_S}{1-\beta}\bigl(x^{1-\beta}-\delta^{1-\beta}\bigr), & 0\le x<\delta,\\[1ex]
        \dfrac{c_S}{1-\beta}\bigl(\delta^{1-\beta}-(-x)^{1-\beta}\bigr), & -\delta<x<0,\\[1ex]
        0, & \text{otherwise}.
    \end{cases}
\end{align*}
It is immediate that
\[
|(L_S\varphi_0)'(0^\pm)|=\infty.
\]

\begin{figure}[h!]
\centering
\includegraphics[scale=0.5]{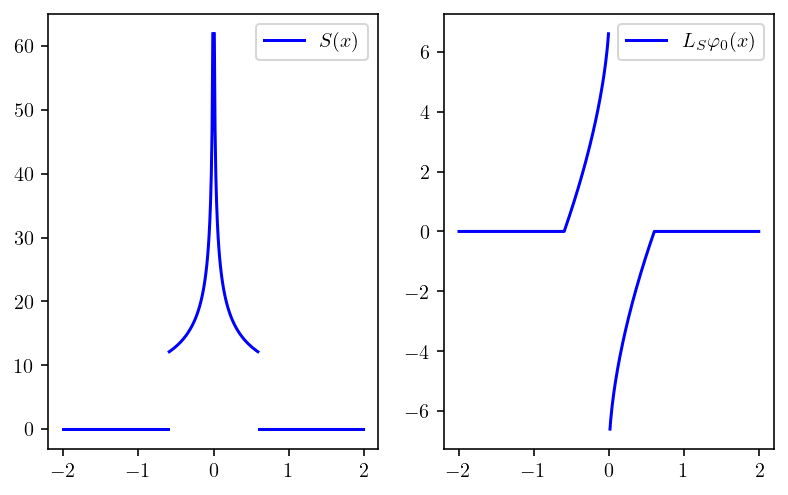}
\caption{Graphs of the kernel \(S(x)\) (left) and \(L_S\varphi_0(x)\) (right) for \(\delta=0.6\) and \(\beta=0.4\).}
\label{fig:kernel_1}
\end{figure}

We next consider kernels that remain bounded at the origin, but whose first derivative is unbounded there. In this case, the singularity is transferred to the second derivative of the solution.

\begin{theorem}\label{thm:K_unbounded1}
Assume that the kernel \(K\) satisfies Assumptions (A1)--(A2), and, in addition, that \(K \in L^\infty(\mathbb{R})\), that \(K'\) is singular at the origin, and that \(K' \in L^2(\mathbb{R})\).
\begin{enumerate}
    \item Suppose that \(f\) satisfies Assumption (B) and \(b\in W^{1,2}(\Omega_I)\). Then, for the solution \(u\in L^2(\Omega\cup\Omega_I)\) of \eqref{eq:Poisson}, one has
    \[
    |u''(0^\pm)|=\infty.
    \]
    \item Suppose that \(u\) satisfies Assumption (B). Then, for \(f:=Lu\), one has
    \[
    |f''(0^\pm)|=\infty.
    \]
\end{enumerate}
\end{theorem}

\begin{proof}
\begin{enumerate}
    \item
    Since \(u\in L^2(\mathbb{R})\) by zero extension and \(K\in L^2(\mathbb{R})\) because \(K\) is bounded and compactly supported, the convolution \(u*K\) is continuous on \(\mathbb{R}\). From the decomposition
    \[
    u*K-\alpha u=f \qquad \text{in } \Omega,
    \]
    it follows that \(u\) is continuous on \(\Omega\setminus\{0\}\), and that \(u\) has a jump discontinuity at \(0\) of magnitude
    \[
    c_0:=[u]_0=-\frac{1}{\alpha}[f]_0,
    \]
    where \(0<|c_0|<\infty\).

    Define
    \begin{equation}\label{eq:u0_1}
        u_0:=u-c_0\varphi_0.
    \end{equation}
    Then \(u_0\in C^0(\Omega)\) and solves
    \begin{align}\label{eq:f0_1}
        \begin{cases}
            Lu_0=f_0:=f-c_0L\varphi_0 & \text{in } \Omega,\\
            u_0=b-c_0\varphi_0 & \text{in } \Omega_I.
        \end{cases}
    \end{align}

    Since \(K\) is continuous on \([-\delta,0)\cup(0,\delta]\), it extends continuously to \([-\delta,0]\) and \([0,\delta]\). Writing
    \[
    (u_0*K)(x)=\int_{x-\delta}^{x+\delta}u_0(y)K(x-y)\,dy
    =\int_{x-\delta}^{x+\delta}g(x,y)\,dy,
    \]
    and applying Lemma~\ref{lemma: Leibniz_easy} to each subinterval, we obtain
    \begin{equation}\label{eq:K_bounded_1}
    \begin{split}
        \frac{d}{dx}\int_{x-\delta}^{x}g(x,y)\,dy
        &=g(x,x)-g(x,x-\delta)+\int_{x-\delta}^{x}\frac{\partial g}{\partial x}(x,y)\,dy\\
        &=u_0(x)K(0^+)-u_0(x-\delta)K(\delta^-)+\int_{x-\delta}^{x}u_0(y)K'(x-y)\,dy,
    \end{split}
    \end{equation}
    and similarly,
    \begin{equation}\label{eq:K_bounded_2}
    \begin{split}
        \frac{d}{dx}\int_{x}^{x+\delta}g(x,y)\,dy
        &=g(x,x+\delta)-g(x,x)+\int_{x}^{x+\delta}\frac{\partial g}{\partial x}(x,y)\,dy\\
        &=u_0(x+\delta)K(-\delta^+)-u_0(x)K(0^-)+\int_{x}^{x+\delta}u_0(y)K'(x-y)\,dy.
    \end{split}
    \end{equation}
    Adding \eqref{eq:K_bounded_1} and \eqref{eq:K_bounded_2}, we find
    \begin{equation}\label{eq:K_bounded_u0}
    \begin{split}
        \frac{d}{dx}(u_0*K)(x)
        &=K(-\delta^+)u_0(x+\delta)-K(\delta^-)u_0(x-\delta)+(K(0^+)-K(0^-))u_0(x)\\
        &\quad +(u_0*K')(x).
    \end{split}
    \end{equation}
    Since \(u_0\in L^2(\mathbb{R})\) and \(K'\in L^2(\mathbb{R})\), the convolution \(u_0*K'\) is continuous on \(\mathbb{R}\). Therefore, \((u_0*K)'\) is continuous on \(\Omega\).

    From \eqref{eq:f0_1} and Lemma~\ref{lemma:Deri_K}, we have
    \[
    f_0'=f'-c_0K
    \]
    almost everywhere. Since \(f'\) is bounded on each side of the origin and \(K\in L^\infty(\mathbb{R})\), it follows that \(0\le |[f_0']_0|<\infty\). Let
    \[
    c_1:=[u_0']_0=-\frac{1}{\alpha}[f_0']_0,
    \]
    so that \(0\le |c_1|<\infty\), and define
    \begin{equation}\label{eq:u1_1}
        u_1:=u_0-c_1\varphi_1.
    \end{equation}
    Then \(u_1\in C^0(\Omega)\), \(u_1\) is continuously differentiable on \(\Omega\setminus\{\pm\delta\}\), and \(u_1\) solves
    \begin{align}\label{eq:u1_solve_thm_unbounded}
        \begin{cases}
            Lu_1=f_1:=f_0-c_1L\varphi_1 & \text{in } \Omega,\\
            u_1=b-c_0\varphi_0-c_1\varphi_1 & \text{in } \Omega_I.
        \end{cases}
    \end{align}

    Applying Lemma~\ref{lemma: Leibniz_easy} once more and proceeding as above, we obtain
    \begin{equation}\label{eq:K_bounded_u1}
    \begin{split}
        \frac{d^2}{dx^2}(u_1*K)(x)
        &=K(-\delta^+)u_1'(x+\delta)-K(\delta^-)u_1'(x-\delta)+(K(0^+)-K(0^-))u_1'(x)\\
        &\quad +\frac{d}{dx}(u_1*K')(x).
    \end{split}
    \end{equation}

    We now show that \(u_1\in W^{1,2}(\Omega\cup\Omega_I)\). Since
    \[
    f_1=f-c_0L\varphi_0-c_1L\varphi_1,
    \]
    and \(f\in L^\infty(\Omega)\) by Assumption (B), while \(L\varphi_0,L\varphi_1\in L^\infty(\Omega)\) by Lemma~\ref{lemma:Deri_K}, it follows that \(f_1\in L^\infty(\Omega)\). Also,
    \[
    f_1'=f'-c_0K-c_1L\varphi_0
    \]
    almost everywhere, so \(f_1'\in L^2(\Omega)\). Hence
    \[
    f_1\in W^{1,2}(\Omega)\cap L^\infty(\Omega).
    \]
    By Theorem~\ref{thm:Radu_mod}, this implies \(u_1\in W^{1,2}(\Omega)\). Since \(b\in W^{1,2}(\Omega_I)\), it follows that
    \[
    u_1\in W^{1,2}(\Omega\cup\Omega_I).
    \]
    Therefore \(u_1'\in L^2(\mathbb{R})\).

    By Lemma~\ref{lemma:Leibniz},
    \[
    \frac{d}{dx}(u_1*K')(x)=u_1'*K'(x)
    \]
    almost everywhere. Since \(u_1'\in L^2(\mathbb{R})\) and \(K'\in L^2(\mathbb{R})\), the convolution \(u_1'*K'\) is continuous on \(\mathbb{R}\). Hence \eqref{eq:K_bounded_u1} implies
    \[
    |(u_1*K)''(0^\pm)|<\infty.
    \]

    On the other hand, by Lemma~\ref{lemma:Deri_K},
    \[
    f_1''=f''-c_0K'-c_1K
    \]
    almost everywhere. Since \(K'\) is unbounded at \(0\) and \(0<|c_0|<\infty\), it follows that
    \[
    |f_1''(0^\pm)|=\infty.
    \]
    Using the decomposition
    \[
    u_1*K-\alpha u_1=f_1,
    \]
    we conclude that
    \[
    |u_1''(0^\pm)|=\infty.
    \]
    Finally, by \eqref{eq:u0_1} and \eqref{eq:u1_1},
    \[
    |u''(0^\pm)|=\infty.
    \]

    \item
    The argument is analogous to that of Part (1). Since \(u*K\) is continuous on \(\mathbb{R}\), the decomposition
    \[
    u*K-\alpha u=f \qquad \text{in } \Omega
    \]
    shows that \(f\) has a jump discontinuity at \(0\) of magnitude
    \[
    [f]_0=-\alpha [u]_0,
    \]
    and is continuous on \(\Omega\setminus\{0\}\).

Let
\begin{align*}
    c_0:=[u]_0=-\frac{1}{\alpha}[f]_0,
\end{align*}
and define $u_0:=u-c_0\varphi_0$. Since $u$ satisfies Assumption (B), it is $C^N$ on $(a_1,0)$ and on $(0,a_2)$, with a single jump discontinuity at $0$. By construction the jump at $0$ is removed. Hence $u_0$ is continuous on $\Omega$ and remains $C^N$ on each of the intervals $(a_1,0)$ and $(0,a_2)$. In particular, $u_0$ is absolutely continuous on every compact subinterval of $\Omega$, and its weak derivative agrees almost everywhere with the classical derivative on $(a_1,0)\cup(0,a_2)$. Since $u_0$ and $u_0'$ are both square-integrable on $\Omega\cup\Omega_I$, it follows that
\[
u_0\in W^{1,2}(\Omega\cup\Omega_I).
\]
Lemma \ref{lemma:Leibniz} then implies $(u_0\ast K)^\prime=u_0^\prime\ast K$ almost everywhere.
    As in \eqref{eq:K_bounded_u0}, we have
    \begin{equation}\label{eq:K_bounded_u0_2}
    \begin{split}
        \frac{d}{dx}(u_0*K)(x)
        &=K(-\delta^+)u_0(x+\delta)-K(\delta^-)u_0(x-\delta)+(K(0^+)-K(0^-))u_0(x)\\
        &\quad +(u_0*K')(x).
    \end{split}
    \end{equation}
    Since \(u_0\in L^2(\mathbb{R})\) and \(K'\in L^2(\mathbb{R})\), the convolution \(u_0*K'\) is continuous on \(\mathbb{R}\). Hence \((u_0*K)'\) is continuous on \(\Omega\).

    Define \(f_0:=f-c_0L\varphi_0\). From
    \[
    u_0*K-\alpha u_0=f_0,
    \]
    it follows that \(f_0'\) has a jump discontinuity at \(0\) of magnitude
    \[
    [f_0']_0=-\alpha [u_0']_0.
    \]
    Let
    \[
    c_1:=[u_0']_0=-\frac{1}{\alpha}[f_0']_0,
    \]
    and define
    \[
    u_1:=u_0-c_1\varphi_1.
    \]
    Since \(0<|c_0|<\infty\), we may again differentiate once more to obtain, as in \eqref{eq:K_bounded_u1},
    \begin{equation}\label{eq:K_bounded_u1_2}
    \begin{split}
        \frac{d^2}{dx^2}(u_1*K)(x)
        &=K(-\delta^+)u_1'(x+\delta)-K(\delta^-)u_1'(x-\delta)+(K(0^+)-K(0^-))u_1'(x)\\
        &\quad +\frac{d}{dx}(u_1*K')(x).
    \end{split}
    \end{equation}
    By Lemma~\ref{lemma:Leibniz},
    \[
    (u_1*K')'=u_1'*K'
    \]
    almost everywhere. Since \(u_1'\in L^2(\mathbb{R})\) and \(K'\in L^2(\mathbb{R})\), the convolution \(u_1'*K'\) is continuous on \(\mathbb{R}\). Therefore, from \eqref{eq:K_bounded_u1_2},
    \[
    |(u_1*K)''(0^\pm)|<\infty.
    \]

    Defining \(f_1:=f_0-c_1L\varphi_1\), the decomposition
    \[
    u_1*K-\alpha u_1=f_1
    \]
    implies that
    \[
    |f_1''(0^\pm)|<\infty.
    \]
    By Lemma~\ref{lemma:Deri_K},
    \[
    f_1''=f''-c_0K'-c_1K
    \]
    almost everywhere. Since \(K'\) is unbounded at \(0\) and \(0<|c_0|<\infty\), it follows that
    \[
    |f''(0^\pm)|=\infty.
    \]
\end{enumerate}
\end{proof}

Theorem~\ref{thm:K_unbounded1}
 is illustrated by the  kernel $P$ given in \eqref{example_kernel_2}, which  is bounded, continuous, and satisfies
\(
|P'(0^\pm)|=\infty.
\)
A direct computation gives
\begin{align*}
    L_P\varphi_0(x)=
    \begin{cases}
        c_P(x-\delta)-\dfrac{3(3+\beta)}{\beta(1+\beta)\delta^{3+\beta}}\bigl(x^{1+\beta}-\delta^{1+\beta}\bigr), & 0\le x<\delta,\\[1ex]
        c_P(x+\delta)-\dfrac{3(3+\beta)}{\beta(1+\beta)\delta^{3+\beta}}\bigl(\delta^{1+\beta}-(-x)^{1+\beta}\bigr), & -\delta<x<0,\\[1ex]
        0, & \text{otherwise}.
    \end{cases}
\end{align*}
It follows that
\[
|(L_P\varphi_0)''(0^\pm)|=\infty.
\]

\begin{figure}[h!]
\centering
\includegraphics[scale=0.5]{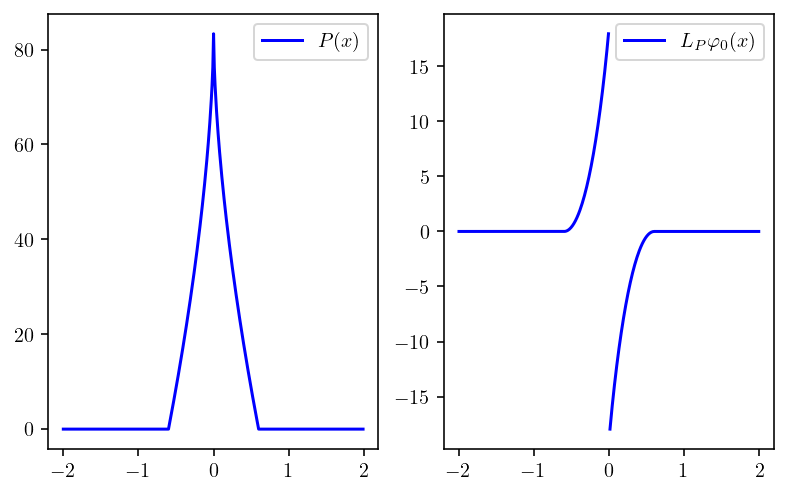}
\caption{Graphs of the kernel \(P(x)\) (left) and \(L_P\varphi_0(x)\) (right) for \(\delta=0.6\) and \(\beta=0.6\).}
\label{fig:kernel_2}
\end{figure}

Theorems~\ref{thm:K_unbounded} and~\ref{thm:K_unbounded1} admit the following generalization.

\begin{theorem}\label{thm:unbounded_gen}
Assume that the kernel \(K\) satisfies Assumptions (A1)--(A2), and, in addition, that
\[
K, K', \ldots, K^{(j-1)} \in L^\infty(\mathbb{R}),
\]
that \(K^{(j)}\) is singular at the origin, and that \(K^{(j)} \in L^2(\mathbb{R})\).
\begin{enumerate}
    \item Suppose that \(f\) satisfies Assumption (B) and \(b\in W^{1,2}(\Omega_I)\). Then, for the solution \(u\in L^2(\Omega\cup\Omega_I)\) of \eqref{eq:Poisson}, one has
    \[
    |u^{(j+1)}(0^\pm)|=\infty.
    \]
    \item Suppose that \(u\) satisfies Assumption (B). Then, for \(f:=Lu\), one has
    \[
    |f^{(j+1)}(0^\pm)|=\infty.
    \]
\end{enumerate}
\end{theorem}

The proof of Theorem~\ref{thm:unbounded_gen} follows by iterating the argument used in Theorem~\ref{thm:K_unbounded1}. At each step, one subtracts the appropriate jump component \(c_m\varphi_m\), differentiates the convolution representation, and uses the singularity of \(K^{(j)}\) at the origin to conclude blow-up of the next derivative.

These results show that the regularity of the kernel at the origin determines the order of derivative blow-up in the solution. In particular, the stronger the singularity of the kernel or its derivatives at the origin, the higher the order at which singular behavior appears in the solution. Thus, for discontinuous source, singular kernels may introduce stronger local irregularities than those present in the data itself, which is an important consideration both for interpretation of the model and for numerical approximation.

\subsection{Jump discontinuities of the kernel at \texorpdfstring{$\pm\delta$}{\(\pm\delta\)}}\label{subsection:jumpdisct}

We next consider the effect of jump discontinuities of the kernel at the horizon endpoints \(\pm\delta\). In contrast to the singular behavior at the origin studied in the previous subsection, these jumps generate a cascade of discontinuities in higher derivatives of the solution at translated locations.

\begin{theorem}\label{thm:K_jump}
Assume that the kernel \(K\) satisfies Assumptions (A1)--(A2), and, in addition, that
\[
K, K', \ldots, K^{(j)}, \ldots \in L^\infty(\mathbb{R}) \quad \text{for all } j,
\]
and that \(K\) has jump discontinuities at \(\pm \delta\).
\begin{enumerate}
    \item Suppose that \(f\) satisfies Assumption (B) and \(b\in W^{1,2}(\Omega_I)\). Then, for the solution \(u\in L^2(\Omega\cup\Omega_I)\) of \eqref{eq:Poisson}, and for every integer \(k\ge 1\) such that the points \(\pm k\delta\) belong to \(\Omega\), one has
    \[
    0<|[u^{(k)}]_{\pm k\delta}|<\infty.
    \]
    
    \item Suppose that \(u\) satisfies Assumption (B). Then, for \(f:=Lu\),
    \[
    0<|[f']_{\pm\delta}|<\infty,
    \]
    and
    \[
    f\in C^\infty(\Omega\setminus\{0,\pm\delta\}).
    \]
\end{enumerate}
\end{theorem}

The jump propagation described below is understood only at translated points that remain inside the spatial domain \(\Omega\). Thus, whenever we refer to jumps at \(\pm k\delta\), we implicitly assume that these points belong to \(\Omega\).
\begin{proof}
\begin{enumerate}
    \item
    Since \(u\in L^2(\mathbb{R})\) by zero extension and \(K\in L^2(\mathbb{R})\) because \(K\) is bounded and compactly supported, the convolution \(u*K\) is continuous on \(\mathbb{R}\). From
    \[
    u*K-\alpha u=f \qquad \text{in } \Omega,
    \]
    it follows that \(u\) has a jump discontinuity at \(0\) of magnitude
    \[
    c_0:=[u]_0=-\frac{1}{\alpha}[f]_0,
    \]
    where \(0<|c_0|<\infty\). Moreover, \(u\) is continuous on \(\Omega\setminus\{0\}\).

    Define
    \begin{equation}\label{eq:K_jump_u0_def}
        u_0:=u-c_0\varphi_0.
    \end{equation}
    Then \(u_0\in C^0(\Omega)\) and \(u_0\) solves
    \[
    \begin{cases}
        Lu_0=f_0:=f-c_0L\varphi_0 & \text{in } \Omega,\\
        u_0=b-c_0\varphi_0 & \text{in } \Omega_I.
    \end{cases}
    \]

    As in \eqref{eq:K_bounded_u0}, applying Lemma~\ref{lemma: Leibniz_easy} yields
    \begin{align}\label{eq:K_jump_u0}
    \begin{split}
        \frac{d}{dx}(u_0*K)(x)
        &=K(-\delta^+)u_0(x+\delta)-K(\delta^-)u_0(x-\delta)+(K(0^+)-K(0^-))u_0(x)\\
        &\quad +(u_0*K')(x).
    \end{split}
    \end{align}
    Since \(u_0\in L^2(\mathbb{R})\) and \(K'\in L^2(\mathbb{R})\), the convolution \(u_0*K'\) is continuous on \(\mathbb{R}\). Hence \((u_0*K)'\) is continuous on \(\Omega\).

    By Lemma~\ref{lemma:Deri_K},
    \[
    f_0'=f'-c_0K
    \]
    almost everywhere. Because \(K\) has jump discontinuities at \(\pm\delta\), it follows that \(f_0'\) has jump discontinuities at \(0\) and \(\pm\delta\). Since \((u_0*K)'\) is continuous, the decomposition
    \[
    u_0*K-\alpha u_0=f_0
    \]
    implies that \(u_0'\) has jump discontinuities at the same points. In particular, using \eqref{eq:K_jump_u0_def}, we obtain
    \[
    0<|[u']_{\pm\delta}|<\infty.
    \]

    Next, let
    \[
    c_1:=[u_0']_0=-\frac{1}{\alpha}[f_0']_0,
    \]
    where \(0\le |c_1|<\infty\), and define
    \begin{equation}\label{eq:K_jump_u1_def}
        u_1:=u_0-c_1\varphi_1.
    \end{equation}
    Then \(u_1\in C^0(\Omega)\), \(u_1\) is continuously differentiable on \(\Omega\setminus\{\pm\delta\}\), and \(u_1\) solves
    \[
    \begin{cases}
        Lu_1=f_1:=f_0-c_1L\varphi_1 & \text{in } \Omega,\\
        u_1=b-c_0\varphi_0-c_1\varphi_1 & \text{in } \Omega_I.
    \end{cases}
    \]

    Applying Lemma~\ref{lemma: Leibniz_easy} once more, as in \eqref{eq:K_bounded_u1}, we obtain
    \begin{align}\label{eq:K_jump_u1}
    \begin{split}
        \frac{d^2}{dx^2}(u_1*K)(x)
        &=K(-\delta^+)u_1'(x+\delta)-K(\delta^-)u_1'(x-\delta)+(K(0^+)-K(0^-))u_1'(x)\\
        &\quad +K'(-\delta^+)u_1(x+\delta)-K'(\delta^-)u_1(x-\delta)+(K'(0^+)-K'(0^-))u_1(x)\\
        &\quad +(u_1*K'')(x).
    \end{split}
    \end{align}
    Since \(u_1\in L^2(\mathbb{R})\) and \(K''\in L^2(\mathbb{R})\), the convolution \(u_1*K''\) is continuous on \(\mathbb{R}\). Moreover, by assumption, \(K(0^+)=K(0^-)\), while \(K(-\delta^+)\) and \(K(\delta^-)\) are nonzero. Since \(u_1\in C^0(\Omega)\), \(u_1'\in C^0(\Omega\setminus\{\pm\delta\})\), and \(u_1'\) has jump discontinuities at \(\pm\delta\), it follows from \eqref{eq:K_jump_u1} that \((u_1*K)''\) has jump discontinuities at \(0\) and \(\pm2\delta\).

    On the other hand, Lemma~\ref{lemma:Deri_K} gives
    \[
    f_1''=f''-c_0K'-c_1K
    \]
    almost everywhere, so \(f_1''\) has jump discontinuities at \(0\) and \(\pm\delta\). Therefore, from the decomposition
    \[
    u_1*K-\alpha u_1=f_1,
    \]
    we conclude that \(u_1''\) has jump discontinuities at \(0\), \(\pm\delta\), and \(\pm2\delta\). Using \eqref{eq:K_jump_u0_def} and \eqref{eq:K_jump_u1_def}, we obtain
    \[
    0<|[u'']_{\pm2\delta}|<\infty.
    \]

    Repeating the same argument inductively shows that each further differentiation produces jump discontinuities at the next translated locations \(\pm k\delta\), which proves the claim.

    \item
    The argument begins as in Part (1). Since \(u*K\) is continuous on \(\mathbb{R}\), the decomposition
    \[
    u*K-\alpha u=f \qquad \text{in } \Omega
    \]
    shows that \(f\) has a jump discontinuity at \(0\) of magnitude
    \[
    [f]_0=-\alpha [u]_0,
    \]
    and that \(f\) is continuous on \(\Omega\setminus\{0\}\).

    Let
    \[
    c_0:=[u]_0=-\frac{1}{\alpha}[f]_0,
    \]
    where \(0<|c_0|<\infty\), and define \(u_0:=u-c_0\varphi_0\). Then \(u_0\in C^0(\Omega)\). As in \eqref{eq:K_jump_u0}, we have
    \begin{align}\label{eq:K_jump_u0_1}
    \begin{split}
        \frac{d}{dx}(u_0*K)(x)
        &=K(-\delta^+)u_0(x+\delta)-K(\delta^-)u_0(x-\delta)+(K(0^+)-K(0^-))u_0(x)\\
        &\quad +(u_0*K')(x).
    \end{split}
    \end{align}
    Since \(u_0\in L^2(\mathbb{R})\) and \(K'\in L^2(\mathbb{R})\), the convolution \(u_0*K'\) is continuous on \(\mathbb{R}\). Hence \((u_0*K)'\) is continuous in \(\Omega\).

    Let \(f_0:=f-c_0L\varphi_0\). From the decomposition
    \[
    u_0*K-\alpha u_0=f_0,
    \]
    it follows that \(f_0'\) has a jump discontinuity at \(0\) of magnitude
    \[
    [f_0']_0=-\alpha [u_0']_0,
    \]
    and \(f_0'\in C^0(\Omega\setminus\{0\})\). By Lemma~\ref{lemma:Deri_K},
    \[
    f_0'=f'-c_0K
    \]
    almost everywhere. Since \(K\) has jumps only at \(\pm\delta\), we conclude that
    \[
    f'\in C^0(\Omega\setminus\{0,\pm\delta\})
    \qquad \text{and} \qquad
    0<|[f']_{\pm\delta}|<\infty.
    \]

    Next, let
    \[
    c_1:=[u_0']_0=-\frac{1}{\alpha}[f_0']_0,
    \]
    where \(0\le |c_1|<\infty\), and define \(u_1:=u_0-c_1\varphi_1\). Then \(u_1\in C^1(\Omega)\). Let \(f_1:=f_0-c_1L\varphi_1\). As in \eqref{eq:K_jump_u1}, we have
    \begin{align}\label{eq:K_jump_u1_1}
    \begin{split}
        \frac{d^2}{dx^2}(u_1*K)(x)
        &=K(-\delta^+)u_1'(x+\delta)-K(\delta^-)u_1'(x-\delta)+(K(0^+)-K(0^-))u_1'(x)\\
        &\quad +K'(-\delta^+)u_1(x+\delta)-K'(\delta^-)u_1(x-\delta)+(K'(0^+)-K'(0^-))u_1(x)\\
        &\quad +(u_1*K'')(x).
    \end{split}
    \end{align}
    Since \(u_1\in L^2(\mathbb{R})\) and \(K''\in L^2(\mathbb{R})\), the convolution \(u_1*K''\) is continuous on \(\mathbb{R}\). Because \(u_1\in C^1(\Omega)\), it follows from \eqref{eq:K_jump_u1_1} that \((u_1*K)''\in C^0(\Omega)\).

    From the decomposition
    \[
    u_1*K-\alpha u_1=f_1,
    \]
    we obtain \(f_1''\in C^0(\Omega\setminus\{0\})\). Using Lemma~\ref{lemma:Deri_K},
    \[
    f_1''=f''-c_0K'-c_1K
    \]
    almost everywhere, which implies
    \[
    f''\in C^0(\Omega\setminus\{0,\pm\delta\}).
    \]
    Repeating the argument inductively yields
    \[
    f\in C^\infty(\Omega\setminus\{0,\pm\delta\}).
    \]
\end{enumerate}
\end{proof}

Theorem~\ref{thm:K_jump} is illustrated by the truncated Gaussian kernel $G$ in \eqref{example_kernel_3},
which is smooth on $\mathbb{R}\setminus\{\pm\delta\}$ and has jump discontinuities at $x=\pm\delta$.
A direct computation gives
\begin{align*}
    L_G\varphi_0(x)=
    \begin{cases}
        \dfrac{-\delta \sqrt{\pi}c_G}{2}\left(\erf\!\left(\frac{-x}{\delta}\right)-\erf(1)\right), & -\delta<x<0,\\[1ex]
        \dfrac{\delta \sqrt{\pi}c_G}{2}\left(\erf\!\left(\frac{x}{\delta}\right)-\erf(1)\right), & 0\le x<\delta,\\[1ex]
        0, & \text{otherwise},
    \end{cases}
\end{align*}
and
\begin{align*}
    (L_G\varphi_0)'(x)=
    \begin{cases}
        c_G\exp\!\left(-\dfrac{x^2}{\delta^2}\right), & -\delta<x<\delta,\; x\neq 0,\\
        0, & x<-\delta \text{ or } x>\delta.
    \end{cases}
\end{align*}
This kernel satisfies the hypotheses of Theorem~\ref{thm:K_jump}. Therefore,
\[
L_G\varphi_0\in C^\infty(\mathbb{R}\setminus\{0,\pm\delta\})
\qquad \text{and} \qquad
0<|[(L_G\varphi_0)']_{\pm\delta}|<\infty.
\]

\begin{figure}[h!]
\centering
\includegraphics[scale=0.5]{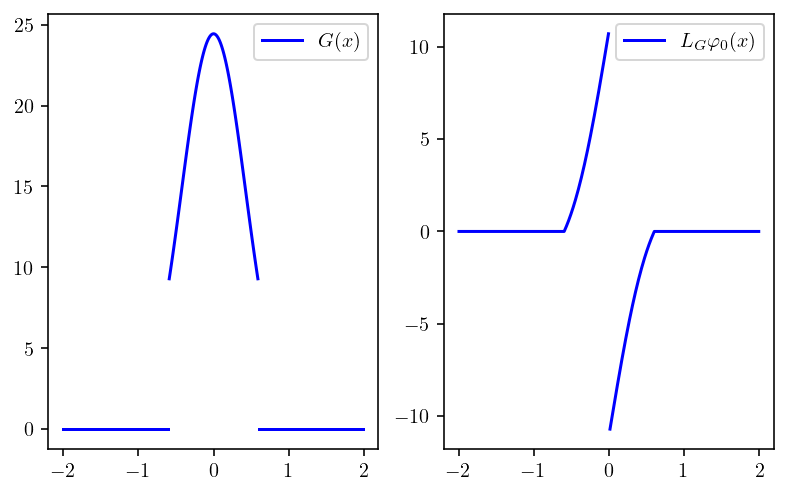}
\caption{Graphs of the kernel \(G(x)\) (left) and \(L_G\varphi_0(x)\) (right) for \(\delta=0.6\).}
\label{fig:kernel_3}
\end{figure}

Theorem~\ref{thm:K_jump} shows that discontinuities of the kernel at the horizon generate a discrete propagation of regularity loss to translated points \(\pm k\delta\). This phenomenon is a direct consequence of the sharp cutoff in the kernel and should be taken into account when assessing whether such kernels are appropriate for a given modeling purpose.

A closely related variant arises when a kernel \(K\) is continuous at \(\pm\delta\), but \(K'\) has jump discontinuities there. In that case, the first nontrivial jump appears one derivative later.

\begin{theorem}\label{thm:K_jump1}
Assume that the kernel \(K\) satisfies Assumptions (A1)--(A2), and, in addition, that
\[
K, K', \ldots, K^{(j)}, \ldots \in L^\infty(\mathbb{R}) \quad \text{for all } j,
\]
that \(K\) is continuous at \(\pm \delta\), and that \(K'\) has jump discontinuities at \(\pm \delta\).
\begin{enumerate}
    \item Suppose that \(f\) satisfies Assumption (B) and \(b\in W^{1,2}(\Omega_I)\). Then, for the solution \(u\in L^2(\Omega\cup\Omega_I)\) of \eqref{eq:Poisson}, and for every integer \(k\ge 1\) such that the points \(\pm k\delta\) belong to \(\Omega\), one has
    \[
    0<|[u^{(k+1)}]_{\pm k\delta}|<\infty.
    \]
    
    \item Suppose that \(u\) satisfies Assumption (B). Then, for \(f:=Lu\),
    \[
    0<|[f'']_{\pm\delta}|<\infty,
    \]
    and
    \[
    f\in C^\infty(\Omega\setminus\{0,\pm\delta\}).
        \]
\end{enumerate}
\end{theorem}

\begin{remark}
The propagation of discontinuities to higher derivatives in nonlocal models was previously observed in \cite{silling2003deformation}, where it is shown (formally, using delta-function arguments) that discontinuities in the solution generate further discontinuities at translated locations, with increasing derivative order. In that setting, the translation distance is determined by the location of discontinuities in the integral kernel and is therefore not restricted to the horizon endpoints.

The results of Theorems~\ref{thm:K_jump} and ~\ref{thm:K_jump1} are consistent with this general mechanism, but provide a rigorous formulation in the present framework. Moreover, our analysis extends directly to kernels with jump discontinuities at arbitrary locations, not only at $\pm\delta$. We focus here on jumps at the horizon endpoints $\pm\delta$, as this is a common feature of many kernels used in applications.
\end{remark}
\subsection{Compactly supported smooth kernel}\label{subsection:smooth}

We finally consider the case of a compactly supported smooth kernel. In this setting, no new singular behavior is created by the operator, and the source term and the solution have equivalent piecewise smooth regularity.

\begin{theorem}\label{thm:K_smooth}
Assume that the kernel \(K\) satisfies Assumptions (A1)--(A2) and, in addition, that \(K\in C_c^\infty(\mathbb{R})\). Let \(f\in L^2(\Omega)\) and \(u\in L^2(\Omega\cup\Omega_I)\) satisfy \(Lu=f\) in \(\Omega\). Then \(f\) satisfies Assumption (B) if and only if \(u\) satisfies Assumption (B). Furthermore,
\begin{align}\label{eq:K_smooth}
    [u^{(k)}]_0=-\frac{1}{\alpha}[f^{(k)}]_0, \qquad k=0,1,2,\ldots
\end{align}
\end{theorem}

\begin{proof}
Since \(K\in C_c^\infty(\mathbb{R})\), the convolution \(u*K\) belongs to \(C^\infty(\mathbb{R})\). The decomposition
\[
u*K-\alpha u=f
\]
therefore implies that \(u\) and \(f\) have the same piecewise smooth regularity. The jump relation \eqref{eq:K_smooth} follows immediately by comparing the left and right limits of the derivatives at \(0\).
\end{proof}

An example of a kernel that statisfies the conditions in Theorem~\ref{thm:K_smooth} is the kernel $Q$ given by \eqref{K_smooth}.

Theorem~\ref{thm:K_smooth} shows that, in contrast to the kernels considered in the previous subsections, a compactly supported \(C^\infty\) kernel does not introduce additional singular behavior into the solution. In particular, for piecewise smooth source, there is no derivative blow-up at the jump location and no cascade of regularity loss at translated points such as \(\pm k\delta\). Thus, the solution inherits exactly the same piecewise smooth structure as the source term, with the jump magnitudes related by \eqref{eq:K_smooth}. From a modeling perspective, this suggests that smooth kernels may be preferable when one wishes to avoid the unintended regularity effects produced by singular kernels or by kernels with jump discontinuities at the horizon.

\section{Semi-analytic spectral method}\label{section:semi-analytic}

In this section, we introduce a semi-analytic spectral method for the accurate resolution of jump discontinuities in the solution of \eqref{eq:Poisson}. The main idea is to transform the original discontinuous problem into a sequence of smoother problems that are better suited for spectral approximation.

We begin with a kernel $K\in C_c^\infty(\mathbb{R})$ satisfying Assumptions~(A1) and~(A2) and
assume that the source term $f$ satisfies Assumption~(B). 
Then, 
 $u$ has a jump at $x=0$ by Theorem~\ref{thm:K_smooth}. 
 
 To simplify notation, throughout this section we suppress the dependence on $0$ and write
  $[\psi]=[\psi]_0$ for the jump at $x=0$.

 By Theorem~\ref{thm:K_smooth}, the solution $u$ also satisfies Assumption~(B), and
\begin{align*}
    [u^{(k)}]=-\frac{1}{\alpha}[f^{(k)}],
\end{align*}
for all $k=0,1,2,\ldots$. Let
\begin{align*}
    c_0:=[u]=-\frac{1}{\alpha}[f],
\end{align*}
and define
\begin{align*}
    u_0:=u-c_0\varphi_0
\end{align*}
on $\mathbb{R}$. Since both $u$ and $\varphi_0$ belong to $C^\infty(\Omega\setminus\{0\})$, it follows that $u_0\in C^\infty(\Omega\setminus\{0\})$. Moreover, the jump at $0$ is removed because
\begin{align*}
    [u_0]=[u]-c_0[\varphi_0]=[u]-c_0=0.
\end{align*}
Hence $u_0\in C(\Omega)$, and therefore $Lu_0=:f_0\in C(\Omega)$ by Lemma~\ref{lemma:Cont}. Since
\begin{align*}
    f_0=f-c_0L\varphi_0,
\end{align*}
Theorem~\ref{thm:K_smooth} and Lemma~\ref{lemma:Deri_K} imply that
\begin{align*}
    c_1:=[u']=[u_0']=-\frac{1}{\alpha}[f_0']=-\frac{1}{\alpha}[f'].
\end{align*}
We then define
\begin{align*}
    u_1:=u_0-c_1\varphi_1
\end{align*}
on $\mathbb{R}$. Since $u_1'=u_0'-c_1\varphi_0$ for all $x\neq 0$, we have $u_1\in C(\Omega)$ and $u_1'\in C(\Omega\setminus\{0\})$. The jump in the derivative is removed because
\begin{align*}
    [u_1']=[u_0']-c_1[\varphi_0]=[u_0']-c_1=0.
\end{align*}
Thus $u_1'\in C(\Omega)$, which implies that both $u_1$ and $f_1:=Lu_1$ belong to $C^1(\Omega)$.

Proceeding inductively, we define the sequences $\{c_k\}$, $\{u_k\}$, and $\{f_k\}$ by
\begin{align}\label{smoothing}
    \begin{cases}
        c_k:=-\dfrac{1}{\alpha}\left[f_{k-1}^{(k)}\right],\\[1mm]
        u_k:=u_{k-1}-c_k\varphi_k,\\[1mm]
        f_k:=f_{k-1}-c_kL\varphi_k,
    \end{cases}
\end{align}
for $k\ge 0$, with $u_{-1}\equiv u$ and $f_{-1}\equiv f$. By Theorem~\ref{thm:K_smooth}, and by repeating the above argument inductively, it follows that
\[
u_k,\; f_k\in C^k(\Omega), \qquad k\ge 0.
\]
This motivates the two successive algorithms, Algorithms~\ref{alg:smooth} and~\ref{alg:resolve}, for resolving the jump discontinuities in $u$.

\begin{algorithm}[H]
\begin{algorithmic}
\Function{smooth}{$f,M$}
\State $f_{-1}\gets f$
\For{$k=0,1,2,\ldots,M$}
    \State Compute $c_k$ in \eqref{smoothing} symbolically
    \Comment Magnitude of $[u^{(k)}]$
    \State Compute $L\varphi_k$ symbolically
    \State $f_k\gets f_{k-1}-c_kL\varphi_k$
    \Comment $f_k\in C^k(\Omega)$
\EndFor
\State \Return $f_M,c_0,\ldots,c_M$
\EndFunction
\end{algorithmic}
\caption{Successive smoothing procedure for constructing the function $f_M$ and the coefficients $c_0,\ldots,c_M$ in \eqref{smoothing}.}
\label{alg:smooth}
\end{algorithm}

\begin{algorithm}[H]
\begin{algorithmic}
\Function{resolve}{$f,b,M$}
\State $f_M,c_0,c_1,\ldots,c_M\gets \textsc{smooth}(f,M)$
\State $b_M\gets b-c_0\varphi_0-c_1\varphi_1-\cdots-c_M\varphi_M$
\State Solve $Lu_M=f_M$ in $\Omega$, \quad $u_M=b_M$ in $\Omega_I$
\Comment Using a spectral solver
\State $u\gets u_M+\sum_{k=0}^M c_k\varphi_k$
\Comment Semi-analytic reconstruction
\State \Return $u$
\EndFunction
\end{algorithmic}
\caption{Construction of the semi-analytic approximation to the solution of \eqref{eq:Poisson}.}
\label{alg:resolve}
\end{algorithm}

One limitation of using kernels in $C_c^\infty(\mathbb{R})$ is that their Fourier multipliers $m$, defined through
\begin{align*}
    \widehat{Lu}=m\widehat{u},
\end{align*}
are often difficult to compute explicitly. As noted in \cite{alali2021fourier}, this multiplier is a central component of the spectral method. As a practical trade-off, we therefore reduce the regularity of $K$ while retaining enough smoothness for Algorithms~\ref{alg:smooth} and~\ref{alg:resolve} to remain valid up to a finite level $M$.

For example, consider the kernel
\begin{align}\label{eq:K_numerical}
    K(x,y)=K(|x-y|)=\frac{105}{\delta^3}\left(1-\left|\frac{x-y}{\delta}\right|\right)^4\chi_{(-\delta,\delta)}(x-y).
\end{align}
Then, $K\in C^3(\mathbb{R}\setminus\{0\})$, and one can verify that Algorithms~\ref{alg:smooth} and~\ref{alg:resolve} apply up to level $M=4$.     Higher levels of smoothing can be achieved by choosing a smoother kernel $K$ by a slight modification to \eqref{eq:K_numerical}.
\begin{figure}[h!]
\centering
\includegraphics[scale=1.0]{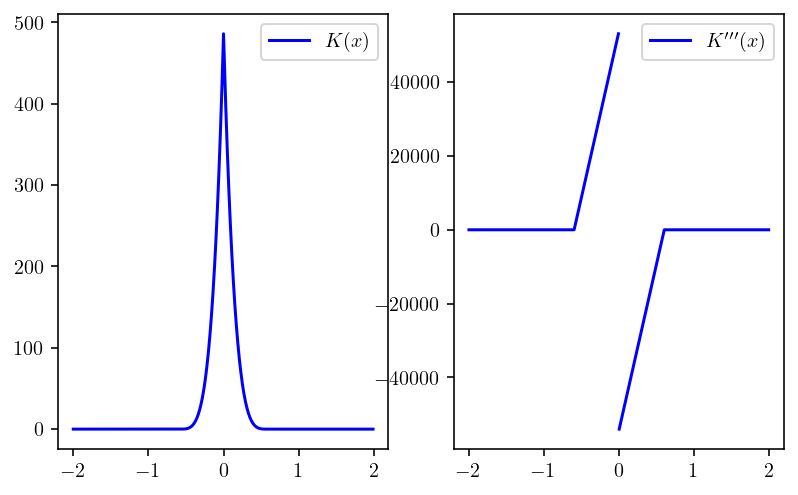}
\caption{Graphs of the kernel \(K(x)\) (left) and \(K^{\prime\prime\prime}(x)\) (right) for \(\delta=0.6\).}
\label{fig:kernel_3}
\end{figure}

\section{Numerical results}\label{section:num}

In this section, we present numerical experiments illustrating the performance of the proposed semi-analytic spectral method. We consider both manufactured solutions, for which the exact solution is known, and a problem with a prescribed discontinuous source term, for which the solution is computed numerically. In all cases, the results demonstrate that successive smoothing significantly improves both accuracy and convergence.

\subsection{Manufactured solutions}

We begin with manufactured solutions $u$ that are continuous across the interface between $\Omega$ and $\Omega_I$. For each such choice of $u$, the corresponding source term $f$ and constraint $b$ in \eqref{eq:Poisson} are generated from the exact solution. More precisely, $f$ is computed symbolically from the definition of the operator $L$ in \eqref{eq:L_operator}, while $b$ is obtained by restricting $u$ to $\Omega_I$.

In Examples~\ref{ex1} and~\ref{ex2}, we take
$
\Omega=(-8,8), \;\; \delta=1.6,
$
and use the kernel $K$ given by \eqref{eq:K_numerical}. We denote by $N$ the number of grid points in $\Omega$, by $M$ the smoothing level used in Algorithms~\ref{alg:smooth} and~\ref{alg:resolve}, and by $u_N^M$ the corresponding numerical approximation.

\begin{example}\label{ex1}
Consider the manufactured solution
\begin{align*}
    u(x)=
    \begin{cases}
        e^x, & x\ge 0,\\
        0, & x<0.
    \end{cases}
\end{align*}
This function satisfies Assumption~(B). The corresponding source term $f=Lu$ is computed symbolically using \textsc{Mathematica}; its explicit formula is provided in Appendix~\ref{appendix}. We then solve \eqref{eq:Poisson} using the semi-analytic spectral method described in Algorithms~\ref{alg:smooth} and~\ref{alg:resolve}.
\end{example}

For each pair $(N,M)$, we define the error by
\[
E_N^M=\|u_{\mathrm{exact}}-u_N^M\|_\infty.
\]
The observed order of convergence is computed from the errors corresponding to consecutive values of $N$.

Table~\ref{tab:example1} and Figure~\ref{fig:ex1} show the effect of smoothing on the numerical approximation. Without smoothing, the method exhibits essentially first-order convergence, reflecting the influence of the jump discontinuity. A substantial improvement is already visible at smoothing level $M=0$, where the method becomes approximately second order. For $M=1$ and $M=2$, the observed convergence rate is close to fourth order, while for $M=3$ and $M=4$ the method exhibits very rapid convergence, with errors reaching near machine precision on relatively modest grids. These results clearly illustrate the effectiveness of the smoothing procedure in restoring high-order accuracy.

\renewcommand{\arraystretch}{1.2}
\begin{table}[h]
    \centering
    \begin{tabularx}{\textwidth}{
  *{5}{| >{\centering\arraybackslash}X}|
  }

\hline
\multirow{2}{*}{$N$} & \multicolumn{2}{c|}{No smoothing} & \multicolumn{2}{c|}{Smoothing level $M=0$} \\ \cline{2-5}
& $L^\infty$ error & Order & $L^\infty$ error & Order \\ \cline{1-5}
$41$  & $12.09$ & -    & $0.84$ & - \\
$61$  & $8.31$  & $0.93$ & $0.38$ & $1.95$ \\
$81$  & $6.33$  & $0.95$ & $0.22$ & $1.97$ \\
$101$ & $5.11$  & $0.96$ & $0.14$ & $1.98$ \\
$121$ & $4.28$  & $0.97$ & $0.10$ & $1.98$ \\
\hline
\hline
\multirow{2}{*}{$N$} & \multicolumn{2}{c|}{Smoothing level $M=1$} & \multicolumn{2}{c|}{Smoothing level $M=2$} \\ \cline{2-5}
& $L^\infty$ error & Order & $L^\infty$ error & Order \\ \cline{1-5}
$41$  & $3.26\cdot 10^{-3}$ & -    & $2.52\cdot 10^{-3}$ & - \\
$61$  & $5.94\cdot 10^{-4}$ & $4.20$ & $4.58\cdot 10^{-4}$ & $4.20$ \\
$81$  & $1.86\cdot 10^{-4}$ & $4.03$ & $1.45\cdot 10^{-4}$ & $3.99$ \\
$101$ & $7.6\cdot 10^{-5}$  & $3.99$ & $6.0\cdot 10^{-5}$  & $3.99$ \\
$121$ & $3.7\cdot 10^{-5}$  & $4.00$ & $2.9\cdot 10^{-5}$  & $3.99$ \\
\hline
\hline
\multirow{2}{*}{$N$} & \multicolumn{2}{c|}{Smoothing level $M=3$} & \multicolumn{2}{c|}{Smoothing level $M=4$} \\ \cline{2-5}
& $L^\infty$ error & Order & $L^\infty$ error & Order \\ \cline{1-5}
$41$  & $9.59\cdot 10^{-4}$ & -     & $9.60\cdot 10^{-4}$ & - \\
$61$  & $1.62\cdot 10^{-5}$ & $10.07$ & $1.62\cdot 10^{-5}$ & $10.07$ \\
$81$  & $8.77\cdot 10^{-7}$ & $10.13$ & $8.81\cdot 10^{-7}$ & $10.12$ \\
$101$ & $9.05\cdot 10^{-8}$ & $10.18$ & $9.14\cdot 10^{-8}$ & $10.15$ \\
$121$ & $1.64\cdot 10^{-8}$ & $9.36$  & $1.43\cdot 10^{-8}$ & $10.16$ \\
\hline
\end{tabularx}
    \caption{Errors and observed convergence rates for Example~\ref{ex1} at different smoothing levels.}
    \label{tab:example1}
\end{table}

\begin{figure}[ht]
\centering
\includegraphics[scale=0.3]{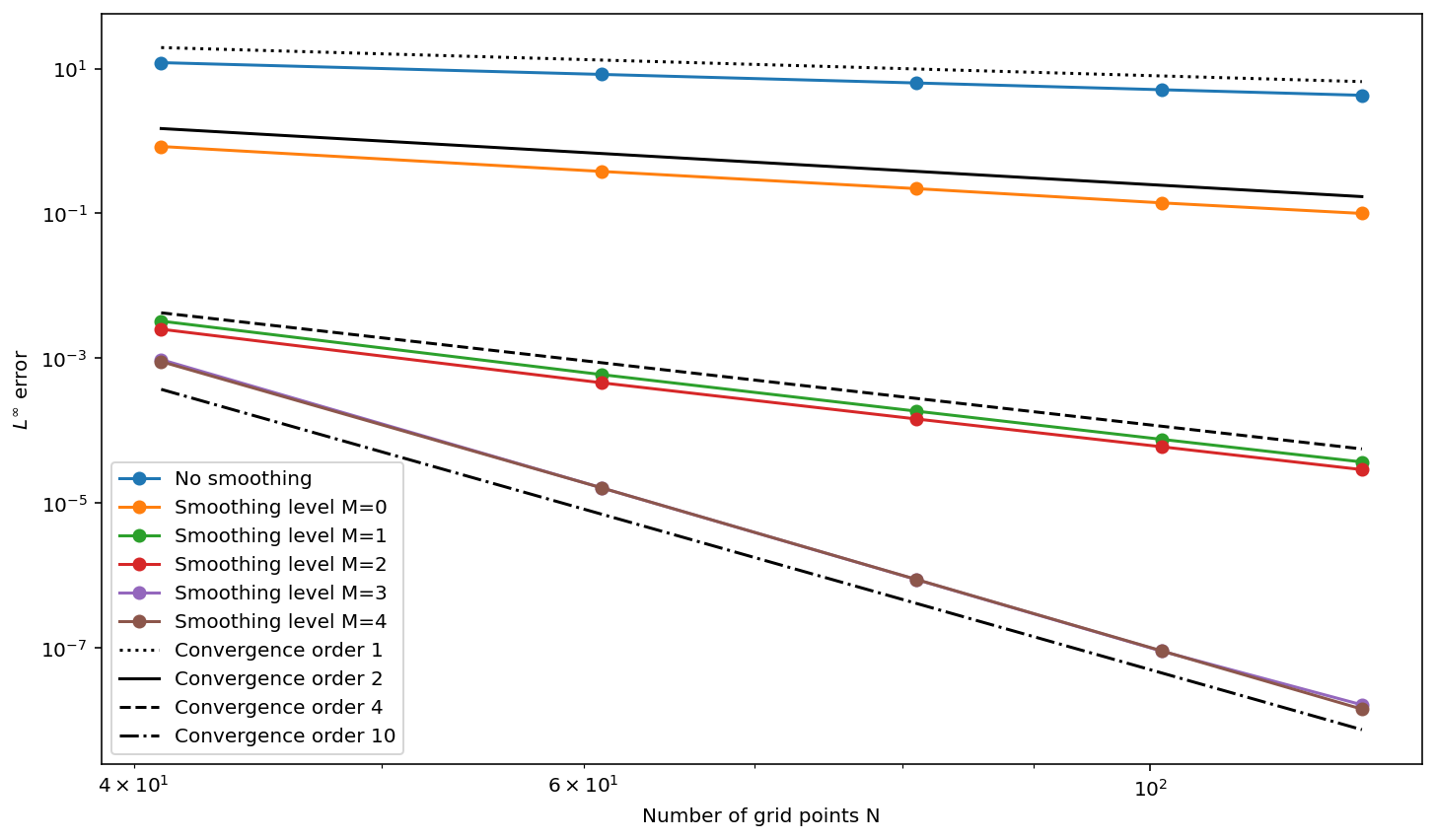}
\caption{Observed convergence rates for Example~\ref{ex1} at different smoothing levels.}
\label{fig:ex1}
\end{figure}

\begin{example}\label{ex2}
Consider the piecewise smooth manufactured solution
\begin{align*}
    u(x)=
    \begin{cases}
        \sin(x+2), & x\le -2,\\
        0, & -2<x<4,\\
        e^{x-4}, & x\ge 4.
    \end{cases}
\end{align*}
This function has jump discontinuities at $x=-2$ and $x=4$. As in Example~\ref{ex1}, the source term $f=Lu$ is computed symbolically using \textsc{Mathematica}, and its explicit expression is listed in Appendix~\ref{appendix}. The semi-analytic spectral method is then applied, with the smoothing procedure carried out separately at each jump.
\end{example}

The corresponding results are presented in Table~\ref{tab:example2} and Figure~\ref{fig:ex2}. The same general pattern observed in Example~\ref{ex1} is present here as well. Without smoothing, the method again behaves essentially as a first-order scheme. Once smoothing is introduced, the accuracy improves markedly. At levels $M=1$ and $M=2$, the method achieves approximately fourth-order convergence, while at levels $M=3$ and $M=4$ the method again displays much faster convergence. These experiments confirm that the proposed approach remains highly effective even in the presence of multiple jump discontinuities.

\renewcommand{\arraystretch}{1.2}
\begin{table}[ht]
    \centering
    \begin{tabularx}{\textwidth}{
  *{5}{| >{\centering\arraybackslash}X}|
  }

\hline
\multirow{2}{*}{$N$} & \multicolumn{2}{c|}{No smoothing} & \multicolumn{2}{c|}{Smoothing level $M=0$} \\ \cline{2-5}
& $L^\infty$ error & Order & $L^\infty$ error & Order \\ \cline{1-5}
$21$  & $16.15$ & -    & $2.92$ & - \\
$41$  & $9.48$  & $0.77$ & $0.47$ & $2.64$ \\
$81$  & $4.88$  & $0.96$ & $0.12$ & $1.97$ \\
$121$ & $3.29$  & $0.98$ & $5.32\cdot 10^{-2}$ & $1.99$ \\
$161$ & $2.48$  & $0.98$ & $3.0\cdot10^{-2}$ & $1.99$ \\
\hline
\hline
\multirow{2}{*}{$N$} & \multicolumn{2}{c|}{Smoothing level $M=1$} & \multicolumn{2}{c|}{Smoothing level $M=2$} \\ \cline{2-5}
& $L^\infty$ error & Order & $L^\infty$ error & Order \\ \cline{1-5}
$21$  & $2.60\cdot 10^{-2}$ & -    & $2.32\cdot 10^{-2}$ & - \\
$41$  & $2.72\cdot 10^{-3}$ & $3.26$ & $3.0\cdot 10^{-3}$ & $2.95$ \\
$81$  & $1.76\cdot 10^{-4}$ & $3.95$ & $1.93\cdot 10^{-4}$ & $3.96$ \\
$121$ & $3.5\cdot 10^{-5}$  & $3.99$ & $3.8\cdot 10^{-5}$ & $3.98$ \\
$161$ & $1.1\cdot 10^{-5}$  & $4.00$ & $1.2\cdot 10^{-5}$ & $3.99$ \\
\hline
\hline
\multirow{2}{*}{$N$} & \multicolumn{2}{c|}{Smoothing level $M=3$} & \multicolumn{2}{c|}{Smoothing level $M=4$} \\ \cline{2-5}
& $L^\infty$ error & Order & $L^\infty$ error & Order \\ \cline{1-5}
$21$  & $1.53\cdot 10^{-2}$ & -    & $1.53\cdot 10^{-2}$ & - \\
$41$  & $1.69\cdot 10^{-5}$ & $9.83$ & $1.72\cdot 10^{-5}$ & $9.80$ \\
$81$  & $1.44\cdot 10^{-7}$ & $6.87$ & $7.98\cdot 10^{-8}$ & $7.75$ \\
$121$ & $1.25\cdot 10^{-8}$ & $6.03$ & $6.87\cdot 10^{-9}$ & $6.05$ \\
$161$ & $2.21\cdot 10^{-9}$ & $6.04$ & $1.27\cdot 10^{-9}$ & $5.88$ \\
\hline
\end{tabularx}
    \caption{Errors and observed convergence rates for Example~\ref{ex2} at different smoothing levels.}
    \label{tab:example2}
\end{table}

\begin{figure}[h!]
\centering
\includegraphics[scale=0.3]{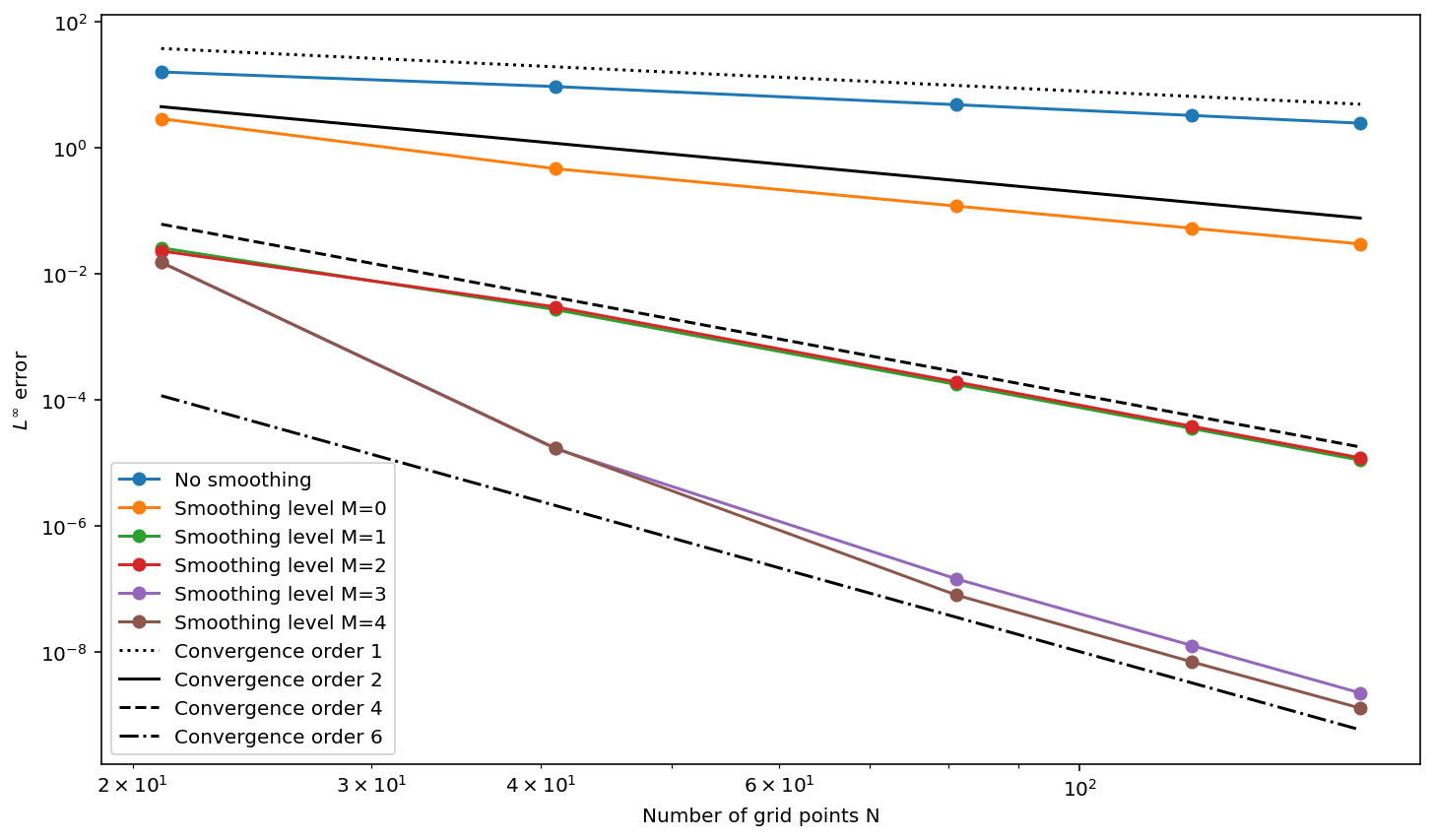}
\caption{Observed convergence rates for Example~\ref{ex2} at different smoothing levels.}
\label{fig:ex2}
\end{figure}

\subsection{Arbitrary source term}

We next consider the case in which the source term $f$ is prescribed directly and satisfies Assumption~(B), rather than being generated from a known exact solution. In this setting, the choice of the constraint $b$ on $\Omega_I$ becomes important. If $b$ is imposed arbitrarily, for example by setting $b\equiv 0$, then the corresponding solution generally develops jump discontinuities at the transition between $\Omega$ and $\Omega_I$; see \cite{mustapha2025fourier,raduneumann2025}. Since Algorithms~\ref{alg:smooth} and~\ref{alg:resolve} are designed to address discontinuities generated by the interior source term, but not those induced at the transition layer, such incompatibilities lead to a noticeable loss of accuracy.

To avoid this issue, we construct a compatible constraint $b$ numerically, as described in Appendix~\ref{appendix}. This enables the smoothing procedure to focus on the interior jumps of the solution and allows the spectral approximation to recover high-order behavior.

\begin{example}\label{ex3}
We consider the source term
\begin{align*}
    f(x)=
    \begin{cases}
        \dfrac{10}{6x^2+0.8}, & x\ge 0,\\
        \tanh(3x)+1, & x<0.
    \end{cases}
\end{align*}
This function satisfies Assumption~(B). We take
$\Omega=(-2,2), \;\; \delta=0.4,
$
and again use the kernel $K$ from \eqref{eq:K_numerical}. We denote 
by $N$ the number of grid points in $\Omega$, and by $M$ the smoothing level used in Algorithms~\ref{alg:smooth} and~\ref{alg:resolve}. The corresponding numerical approximation is denoted 
by $u_N^M$.
\end{example}

Since no exact solution is available in this case, we use the following error measurement strategy. In Table~\ref{tab:example3.3}, we compute a reference solution $u_{\mathrm{approx}}$ on a fine grid with $N=401$ points in $\Omega$, and define
\[
E_N^M=\|u_N^M-u_{\mathrm{approx}}\|_\infty,
\]
where the difference is taken after restricting $u_{\mathrm{approx}}$
 to the coarse grid on which both functions are defined. No interpolation is required, since the grids are nested.
The associated convergence rate is estimated by
\[
\frac{\log(E_{N_1}^M/E_{N_2}^M)}{\log(N_2/N_1)},
\]
where $N_1$ and $N_2$ are two consecutive values of $N$ listed in the table.

We observe the following trend: higher smoothing levels lead to smaller errors and significantly improved convergence. In particular, the transition from no smoothing to even a modest amount of smoothing produces a substantial gain, while higher smoothing levels restore the high-order accuracy expected from the spectral solver once the dominant discontinuity effects have been removed.

\begin{table}[h]
    \centering
    \begin{tabularx}{\textwidth}{
  *{5}{| >{\centering\arraybackslash}X}|
  }

\hline
\multirow{2}{*}{$N$} & \multicolumn{2}{c|}{No smoothing} & \multicolumn{2}{c|}{Smoothing level $M=0$} \\ \cline{2-5}
& $L^\infty$ error & Order & $L^\infty$ error & Order \\ \cline{1-5}
$21$  & $1.124$ & -    & $1.48\cdot 10^{-2}$ & - \\
$41$  & $0.569$ & $1.02$ & $2.72\cdot 10^{-3}$ & $2.53$ \\
$51$  & $0.457$ & $1.01$ & $1.68\cdot 10^{-3}$ & $2.22$ \\
$81$  & $0.287$ & $1.00$ & $6.33\cdot 10^{-4}$ & $2.11$ \\
$101$ & $0.230$ & $1.00$ & $4.03\cdot 10^{-4}$ & $2.05$ \\
\hline
\hline
\multirow{2}{*}{$N$} & \multicolumn{2}{c|}{Smoothing level $M=1$} & \multicolumn{2}{c|}{Smoothing level $M=2$} \\ \cline{2-5}
& $L^\infty$ error & Order & $L^\infty$ error & Order \\ \cline{1-5}
$21$  & $6.57\cdot 10^{-3}$ & -    & $2.05\cdot 10^{-3}$ & - \\
$41$  & $3.81\cdot 10^{-4}$ & $4.26$ & $7.21\cdot 10^{-5}$ & $5.00$ \\
$51$  & $1.49\cdot 10^{-4}$ & $4.29$ & $2.74\cdot 10^{-5}$ & $4.44$ \\
$81$  & $2.27\cdot 10^{-5}$ & $4.07$ & $3.87\cdot 10^{-6}$ & $4.23$ \\
$101$ & $9.23\cdot 10^{-6}$ & $4.08$ & $1.56\cdot 10^{-6}$ & $4.11$ \\
\hline
\hline
\multirow{2}{*}{$N$} & \multicolumn{2}{c|}{Smoothing level $M=3$} & \multicolumn{2}{c|}{Smoothing level $M=4$} \\ \cline{2-5}
& $L^\infty$ error & Order & $L^\infty$ error & Order \\ \cline{1-5}
$21$  & $1.28\cdot 10^{-3}$ & -    & $3.89\cdot 10^{-4}$ & - \\
$41$  & $1.66\cdot 10^{-5}$ & $6.50$ & $1.90\cdot 10^{-6}$ & $7.95$ \\
$51$  & $4.12\cdot 10^{-6}$ & $6.37$ & $4.37\cdot 10^{-7}$ & $6.73$ \\
$81$  & $2.21\cdot 10^{-7}$ & $6.32$ & $2.79\cdot 10^{-8}$ & $5.95$ \\
$101$ & $5.46\cdot 10^{-8}$ & $6.34$ & $8.02\cdot 10^{-9}$ & $5.65$ \\
\hline
\end{tabularx}
    \caption{Errors and observed convergence rates for Example~\ref{ex3}, using a reference solution computed on a fine grid.}
    \label{tab:example3.3}
\end{table}

\section{Conclusion}

In this work, we studied a one-dimensional nonlocal Poisson problem with discontinuous source term and developed a unified analytical and computational framework for understanding and resolving the resulting discontinuities. On the analytical side, we showed that the discontinuity structure of the solution is governed primarily by the regularity of the interaction kernel. Under general assumptions on compactly supported integrable kernels, jump discontinuities in the source term are inherited by the solution, while the finer behavior of the solution near those jumps is dictated by the behavior of the kernel at the origin and at the horizon endpoints.

More precisely, singularities of the kernel, or of its derivatives, at the origin lead to blow-up of corresponding derivatives of the solution at the source discontinuity. By contrast, jump discontinuities of the kernel, or of its derivatives, at \(\pm\delta\) generate cascades of derivative jumps at translated locations inside the domain. These effects are not merely technical possibilities: they arise for kernels commonly used in peridynamic-type models. At the same time, they show that certain kernel choices may introduce stronger local irregularities than those present in the data itself, as well as discrete losses of regularity at points away from the original jump. In this sense, the analysis reveals potentially unintended consequences of singular kernels and of kernels with sharp horizon cutoffs.

In the smooth compactly supported case, no additional singular structure is created by the operator, and the source term and the solution have equivalent piecewise smooth regularity. 
Thus, smooth kernels avoid the additional regularity effects associated with singular kernels or kernels that are discontinuous at the horizon, which may make them preferable in modeling situations where such effects are not physically desired.

Motivated by this analysis, we then introduced a semi-analytic spectral method for the accurate numerical treatment of discontinuous nonlocal problems. The method is based on successive smoothing transformations that explicitly remove jump discontinuities through analytically constructed correction functions. This converts the original problem into an auxiliary problem with improved regularity, allowing a spectral solver to operate in a regime where high-order accuracy can be recovered. The approximation to the original discontinuous solution is then reconstructed by combining the numerical solution of the smoothed problem with the analytic correction terms.

The numerical experiments demonstrate that this strategy leads to substantial gains in both accuracy and convergence. Without smoothing, the spectral approximation is limited by discontinuities and exhibits the expected deterioration associated with Gibbs oscillations. Once the dominant discontinuity contributions are removed, the method recovers much higher accuracy and, in many cases, high-order convergence. 

Overall, the results of this paper show that regularity analysis and numerical design can be combined in a mutually reinforcing way for nonlocal problems with discontinuities. The analytical theory identifies the singular structures created by the kernel and the source, while the computational framework uses this information to build effective correction terms that preserve the efficiency of spectral solvers while overcoming their principal limitation in the presence of jumps. More broadly, the results clarify how modeling choices at the level of the kernel influence both the qualitative behavior of solutions and the performance of numerical methods.

Several directions remain for future work. These include extending the framework to higher spatial dimensions, treating singular kernels and different boundary constraints, and developing analogous methods for time-dependent nonlocal and peridynamic models involving evolving fracture. Another natural direction is the treatment of multiple interacting discontinuities and more general classes of nonlocal operators arising in peridynamics and related applications.

\appendix
\section{Details for the construction of source terms and constraints}
\label{appendix}

For completeness, we record here the explicit source terms used in the manufactured-solution experiments in Examples~\ref{ex1} and ~\ref{ex2}, as well as the construction of the compatible constraint $b$ used in Example~\ref{ex3}.

In Example~\ref{ex1}, the manufactured source term is
\begin{align*}
    f(x)=
    \begin{cases}
        0, & x\le-\delta,\\
        -\dfrac{105}{\delta^7}(24-24e^{x+\delta}+24x+12x^2+4x^3+x^4+24
        \delta+24x\delta+12x^2\delta+4x^3\delta\\
        \qquad +12\delta^2+12x\delta^2+6x^2\delta^2+4\delta^3+4x\delta^3+\delta^4), & -\delta<x<0,\\
        -\dfrac{21}{\delta^7}(120-120e^{x+\delta}+120x+60x^2+20x^3+5x^4- 120\delta+240\delta e^x-120\delta x\\
        \qquad -60x^2\delta-20x^3\delta+60\delta^2+60x\delta^2+30x^2\delta^2-20\delta^3+40\delta^3 e^x-20x\delta^3\\
        \qquad+5\delta^4+2\delta^5e^x), & 0\le x<\delta,\\
        \dfrac{42}{\delta^7}e^{x-\delta}(-60+60e^{2\delta}-120e^\delta-20\delta^3 e^\delta-\delta^5e^\delta), & x\ge\delta.
    \end{cases}
\end{align*}

In Example~\ref{ex2}, assuming $\delta<3$, the manufactured source term is
\begin{align*}
    f(x)=
    \begin{cases}
        -\dfrac{42}{\delta^7}(-60\cos(2+x-\delta)+60\cos(2+x+\delta)+120\delta\sin(2+x)\\
        \qquad-20\delta^3\sin(2+x)+\delta^5\sin(2+x)),&x\le-2-\delta,\\
        -\dfrac{21}{\delta^7}(-40+80x+60x^2+40x^3+5x^4-80\delta+120\delta x +120\delta x^2+20\delta x^3\\
        \qquad + 60  \delta ^2 + 120 \delta ^2x +
30 \delta ^ 2x^2 + 40\delta^3 + 20\delta^3x + 5\delta^4
-120 \cos(2 + x -\delta)\\
\qquad + 240\delta\sin(2 + x) -40\delta^3\sin(2 + x) + 2 \delta^5
\sin(2 + x)), &-2-\delta<x\le-2,\\
-\dfrac{105}{\delta^7}(-8 -16x + 12x^2 + 8x^3 +
x^4 + 16\delta -24 x\delta -24x^2\delta -4 x^3\delta\\
\qquad + 12 \delta^2 + 24x\delta^2 + 6x^2
\delta^2 -8 \delta^3 -4x\delta^3 + \delta^4
-24\cos(2 + x -\delta)), &-2<x<-2+\delta,\\
0,&-2+\delta\le x\le 4-\delta,\\
-\dfrac{105}{e^4\delta^7}(120e^4 -24
e^{x + \delta} -136e^4 x + 60e^4 x^2 -12e^4 x^3 + e^4 x^4 -136 e^4\delta\\
\qquad + 120 e^4x\delta -36 e^4
x^2 \delta + 4 e^4 x^3\delta + 60e^4\delta^2 -36e^4 x\delta^2+ 6 e^4x^2 \delta ^2\\
\qquad -12 e^4 \delta^ 3
+ 4 e^4 x \delta^3 + e^4\delta^4), &4-\delta<x<4,\\
-\dfrac{21}{e^4\delta^7}(600e^4 -120
e^{x + \delta} -680 e^4 x + 300 e^4 x^2 -60e^4 x^3+ 5 e^4 x^4\\
\qquad+ 680 e^4 \delta + 240 e^x \delta -600e^4 x\delta + 180 e^4 x^2 \delta -20 e^4
x^3 \delta + 300e^4 \delta^2\\
\qquad -180 e^4 x \delta^2 + 30 e^4 x^2 \delta^2 + 60e^4\delta^3 + 40e^x\delta^3 -20e^4 x\delta^3 + 5e^4\delta^4 + 2e^x\delta^5), &4\le x<4+\delta,\\
\dfrac{42}{\delta^7}e^{x-\delta-4}(-60 + 60e^{2\delta} -120\delta e^\delta-20\delta^3e^\delta -\delta^5e^\delta),&x\ge 4+\delta.
    \end{cases}
\end{align*}

In Example~\ref{ex3}, the compatible constraint $b$ is constructed numerically as follows. Let $f_4$ denote the level-$4$ smoothing of $f$ produced by Algorithm~\ref{alg:smooth}, with corresponding coefficients $c_0,c_1,c_2,c_3,c_4$. We enlarge the original domain in \eqref{eq:Poisson} to
\[
\Omega'=(-2-\delta,2+\delta),
\]
with associated collar
\[
\Omega_I'=(-2-2\delta,-2-\delta)\cup(2+\delta,2+2\delta).
\]
The set $\Omega'\cup\Omega_I'$ is then extended periodically, and the problem
\[
Lw_4=f_4
\]
is solved on the torus $\mathbb{T}$. By \cite{alali2021fourier}, since $f_4\in C^4(\mathbb{T})$, the corresponding solution satisfies $w_4\in C(\mathbb{T})$. We then define
\[
w=w_4+\sum_{k=0}^4 c_k\varphi_k,
\]
and set
\[
b=w|_{\Omega_I}.
\]


\bibliographystyle{acm}
\bibliography{ref}

\end{document}